\documentclass[submit,hidelinks,onefignum,onetabnum]{siamart251216}

\usepackage{lipsum}
\usepackage{amsfonts}
\usepackage{graphicx}
\usepackage{epstopdf}
\usepackage{algorithmic}
\usepackage{algorithm}
\usepackage{multirow}
\usepackage{booktabs}
\ifpdf
  \DeclareGraphicsExtensions{.eps,.pdf,.png,.jpg}
\else
  \DeclareGraphicsExtensions{.eps}
\fi

\usepackage{xr-hyper}

\makeatletter
\@mparswitchfalse        
\makeatother
\normalmarginpar         

\theoremstyle{plain}
\newsiamremark{remark}{Remark}
\newsiamremark{hypothesis}{Hypothesis}
\crefname{hypothesis}{Hypothesis}{Hypotheses}
\newsiamthm{claim}{Claim}
\newsiamremark{fact}{Fact}
\crefname{fact}{Fact}{Facts}

\newtheorem{pr}[theorem]{Problem}
\numberwithin{equation}{section}

\newtheorem{as}[theorem]{Assumption}

\DeclareMathOperator*{\arginf}{arg\,inf}
\DeclareMathOperator*{\argmax}{arg\,max}
\DeclareMathOperator*{\esssup}{ess\,sup}
\DeclareMathOperator*{\essinf}{ess\,inf}
\DeclareMathOperator*{\Lip}{Lip}

\def\P{\mathbb{P}}

\def\E{\mathbb{E}}
\def\H{\mathrm{H}}

\def\calT{\mathcal{T}}

\def\N{\overline{N}}

\def\J{\mathcal{J}}

\def\K{\overline{K}}

\def\R{\mathbb{R}}

\def\Var{\mathrm{Var}}

\DeclareMathOperator*{\size}{size}
\DeclareMathOperator*{\Growth}{Gr}
\DeclareMathOperator*{\Holdertwo}{\mathrm{Hol}_2}

\DeclareMathOperator*{\LipH}{LipH}

\allowdisplaybreaks

\makeatletter
\DeclareRobustCommand*\cal{\@fontswitch\relax\mathcal}
\makeatother

\usepackage{amssymb}

\headers{Dual optimal switching and DeepMartingale}{J. Ye and H.\;Y.\;Wong}

\title{
Duality and DeepMartingale for High-Dimensional Optimal Switching: Computable Upper Bounds and Approximation-Expressivity Guarantees
\thanks{Submitted to the editors \today.
\funding{H.\;Y.\;Wong acknowledges the support from the Research Grants Council of Hong Kong (grant DOI:  GRF14308422).}}}

\author{Junyan Ye\thanks{Department of Statistics and Data Science, The Chinese University of Hong Kong, Shatin, N.T., Hong Kong 
  (\email{junyanye@link.cuhk.edu.hk}, \email{hywong@cuhk.edu.hk}).}
\and Hoi Ying Wong\footnotemark[2]
  }

\usepackage{amsopn}
\DeclareMathOperator{\diag}{diag}

\ifpdf
\hypersetup{
  pdftitle={Dual optimal switching and DeepMartingale},
  pdfauthor={J. Ye and H. Y. Wong}
}
\fi

\begin{document}

\maketitle

\begin{abstract}
We study finite-horizon optimal switching with discrete intervention dates on a general filtration, allowing continuous-time observations between decision dates, and develop a deep-learning-based dual framework with computable upper bounds. We first derive a dual representation for multiple switching by introducing a family of martingale penalties. The minimal penalty is characterized by the Doob martingales of the continuation values, which yields a fully computable upper bound. We then extend DeepMartingale from optimal stopping to optimal switching and establish convergence under both the upper-bound loss and an $L^2$-surrogate loss. We also provide an expressivity analysis: under the stated structural assumptions, for any target accuracy $\varepsilon>0$, there exist neural networks of size at most $c d^{q}\varepsilon^{-r}$ whose induced dual upper bound approximates the true value within $\varepsilon$, where $c$, $q$, and $r$ are independent of $d$ and $\varepsilon$. Hence,
the dual solver avoids the curse of dimensionality under the stated structural assumptions. For numerical assessment, we additionally implement a deep policy-based approach to produce feasible lower bounds and empirical upper--lower gaps. Numerical experiments on Brownian and Brownian--Poisson models demonstrate small upper--lower gaps and favorable performance in high dimensions. The learned dual martingale also yields a practical delta-hedging strategy.
\end{abstract}



\section{Introduction}\label{sec:intro}

Optimal switching concerns sequential regime changes under uncertainty when each switch incurs a cost. In the finite-horizon Markovian continuous-intervention setting, it is classically linked to coupled obstacle/QVI/PIDE systems. We study a discrete-intervention formulation on a general filtration with continuous-time observations between intervention dates. This covers both intrinsically discrete-time models and discretely exercisable continuous-time models; in the Markovian case, it can also be viewed as a grid-restricted approximation of the continuous problem, equivalently as a time-discrete obstacle recursion. The problem is computationally difficult in high dimension because, at each decision date, one must optimize jointly over intervention and post-intervention regime, while continuation values remain coupled across regimes. Applications include natural-resource management \cite{Bernan-Schwartz-85}, firm entry--exit \cite{Dixit-89}, energy and electricity problems such as tolling, storage, and scheduling \cite{Carmona-Ludkoviski01122008,Carmona-Ludkovski-10,Bayraktar23-deep-switching}.

Optimal switching has been studied via PDE/ODE methods \cite{Oksendal-Brekke-94,Duckworth-zervos-00,zervos-98,Pham-switch-07} and BSDE methods \cite{Hamadene-07,Hamadene-Djehiche-09}. Explicit solutions are unavailable except in special cases, so numerical methods are essential. Grid-based QVI/PIDE solvers are effective only in low dimension. Regression-based dynamic programming \cite{ludkovski2005-thesis,Carmona-Ludkoviski01122008,Carmona-Ludkovski-10,Pham-sifin-14} alleviates but does not remove the curse of dimensionality, since the approximation space still grows rapidly with dimension. More recent deep-learning methods for high-dimensional PDEs and BSDEs \cite{han-weinanE18,RAISSI2019-pinns,pham2020deepBSDE,Becker-Jentzen-Neufeld-Deep-Splitting-21,pham2022deepBSDE_erroranalysis}, including switching with jumps via reflected BSDEJs \cite{Bayraktar23-deep-switching}, scale better empirically. Most such deep solvers, however, are value-based: they approximate value/continuation functions and recover the switching rule by comparison, but typically do not provide computable genuine upper bounds or switching-specific high-dimensional approximation guarantees.

A complementary direction is policy-based learning, where the control is parameterized directly. For optimal stopping this was introduced in \cite{Becker19} and extended to multiple stopping \cite{HAN2023106881} and impulse control \cite{Jia-Wong01022024}. 
Primal methods naturally yield feasible controls and hence lower bounds, but computable upper bounds are usually unavailable. 
In the present paper, however, our primary goal is not to develop a new primal learning theory, but rather to obtain computable genuine upper bounds for high-dimensional optimal switching.

For optimal stopping, martingale duality provides a natural route to genuine upper bounds \cite{Roger02,Haugh04,belome09,schoen13}. Building on this literature and recent neural approximation results \cite{Jentzen20,Jentzen23,gonon23}, Ye and Wong \cite{ye2025deepmartingale} introduced \emph{DeepMartingale} for discrete stopping with continuous-time observations, together with rigorous approximation guarantees. A main goal of the present paper is to extend this dual viewpoint from stopping to optimal switching on a general filtration. This is nontrivial: the controller must choose both intervention times and post-intervention regimes, and the continuation values are coupled across regimes, so classical stopping duality does not directly yield a computable dual formulation for switching. 
Moreover, the approximation-expressivity analysis framework can not be directly applied by \cite{ye2025deepmartingale} due to the existence of continuous-time integral of running payoff, as well as the maximum operator of multiple switching-regimes.

To our knowledge, a fully computable martingale-dual theory for finite-horizon optimal switching is still missing. The closest related work is Lin and Ludkovski \cite{lin2009dual}, where the upper bound still depends on the unknown value function and is therefore not fully computable. 
We therefore develop a deep-learning-based dual method for high-dimensional switching that produces computable genuine upper bounds.
For numerical benchmarking, we additionally establish primal dynamic programming principle and implement a deep policy-based approach to compute feasible lower bounds and report empirical upper--lower gaps. 
The learned dual martingale also admits a natural hedging interpretation.
From the viewpoint of scientific computing, the central challenge is to combine scalability, computable dual upper bounds, and high-dimensional approximation theory within one framework, 
while assessing the resulting dual solver against feasible lower bounds.

\medskip
\noindent\textbf{Our main contributions are:}
\begin{itemize}
    \item[(i)] We derive a martingale-dual representation for finite-horizon optimal switching with discrete intervention dates. Via an equivalent regime-decision reformulation, we prove strong duality and obtain fully computable genuine upper bounds.

    \item[(ii)] We extend \emph{DeepMartingale} \cite{ye2025deepmartingale} from stopping to switching and analyze the resulting solver. We prove convergence under both the upper-bound loss and an \(L^2\)-surrogate loss, and establish approximation/expressivity results that, under the stated structural assumptions, avoid the curse of dimensionality. We also instantiate the theory for affine It\^o diffusions.

    \item[(iii)] We implement the dual solver in Brownian and Brownian--Poisson settings. For numerical benchmarking, we additionally compute feasible lower bounds through primal dynamic programming principle and a deep policy-based approach, which allows us to report empirical upper--lower gaps in practice.
\end{itemize}


The paper is organized as follows. Section~\ref{sec:pb_formulate} formulates the switching problem and its regime-decision reformulation. Section~\ref{sec:duality_equiv} develops the martingale-duality theory and proves the dual dynamic programming principle and strong duality. Section~\ref{sec:DeepSwitchMart} develops the DeepMartingale dual solver, together with its convergence, expressivity, and delta-hedging interpretation in the Brownian Markovian setting. Section~\ref{sec:numerical} reports the numerical experiments.


\subsection{Notations}\label{sec:notations}

Fix \(T>0\) and a filtered probability space \((\Omega,\mathcal F,\mathbb F,\mathbb P)\), where \(\mathbb F=(\mathcal F_t)_{t\in[0,T]}\). For \(t\in[0,T]\), write \(\mathbb F_t:=(\mathcal F_s)_{s\in[t,T]}\). Fix \(N\in\mathbb N_+\) and the uniform grid
\(
\pi: t_n={nT}/{N},\;  n\in \overline N:=\{0,\ldots,N\}.
\)
Set
\(
\N^{-1}:=\overline N\setminus\{N\},\;
\overline N_n:=\{n,\ldots,N\},\;
\N_n^{-1}:=\overline N_n\setminus\{N\},
\)
and define the discrete filtrations
\(
\mathbb F^N:=(\mathcal F_{t_n})_{n\in\overline N},
\)
\(
\mathbb F_n:=(\mathcal F_{t_m})_{m\in\overline N_n}.
\)
We write \(\E[\cdot]\) for expectation and \(\E_t[\cdot]:=\E[\cdot\mid\mathcal F_t]\) for conditional expectation. For vectors, \(\|\cdot\|\) denotes the Euclidean norm; for matrices, \(\|\cdot\|_{\mathrm H}\) denotes the Hilbert--Schmidt norm. We also use convention \(\inf\varnothing:=+\infty\) and \(\sum_{k\in\varnothing}:=0\).

For \(p\ge 1\), \(k\in\mathbb N_+\), and \(0\le s\le t\le T\), let
\[
L^p(\mathcal F_t;\R^k)
:=
\{\xi:\xi \text{ is } \R^k\text{-valued } \mathcal F_t\text{-measurable and } \|\xi\|^p_{L^p}:= \E[\|\xi\|^p]<\infty\}.
\]
Let \(L_N^p(\R^k)\) and \(L_{n,N}^p(\R^k)\) denote the spaces of \(\R^k\)-valued, \(\mathbb F^N\)-adapted and \(\mathbb F_n\)-adapted discrete-time processes with \(L^p(\mathcal F_t;\R^k)\)-integrable components. Likewise, \(\mathbb L^p(\R^k)\) and \(\mathbb L_{s,t}^p(\R^k)\) denote the spaces of \(\R^k\)-valued, \(\mathbb F\)-adapted processes such that
\[
\|Z\|_{\mathbb L^p}^p:=\E\int_0^T \|Z_u\|^p\,du<\infty,
\qquad
\|Z\|_{\mathbb L^p_{s,t}}^p:=\E\int_s^t \|Z_u\|^p\,du<\infty.
\]

If \(\rho\) is a finite Borel measure on \(\R^{k_1}\), define
\[
L^p_{k_1,k_2}(\rho)
:=
\Bigl\{F:\R^{k_1}\to\R^{k_2}:
\|F\|_{p,\rho}^p:=\int_{\R^{k_1}}\|F(x)\|^p\,\rho(dx)<\infty
\Bigr\},
\]
and
\(
\mathbb M_p(\rho):=
\bigl(\int_{\R^{k_1}}\|x\|^p\,\rho(dx)\bigr)^{1/p}.
\)

Let \(J\in\mathbb N_+\) and \(\mathcal J:=\{1,\ldots,J\}\), with \(\mathcal J^{-i}:=\mathcal J\setminus\{i\}\). For \(n\in\N^{-1}\), let
\[
\mathcal J_n:=\{(j_m)_{m=n}^N:\ j_m\in\mathcal J,\ j_N=j_{N-1}\}, \quad \mathcal J_N:=\{(i)\}
\]
and, for \(i\in\mathcal J\), \( n\in \N \), 
\(
\mathcal J_n^i:=\{(j_m)_{m=n-1}^N:\ j_{n-1}=i,\ (j_m)_{m=n}^N\in\mathcal J_n\} .
\)
Let \(\mathcal T^N\) be the set of \(\mathbb F^N\)-stopping times taking values in \(\overline N\), and \(\mathcal T_n\) the set of \(\mathbb F_n\)-stopping times taking values in \(\overline N_n\). Finally, \(\mathcal M_N\) and \(\mathcal M_{n,N}\) denote the sets of $L^1_N(\R^1)$ and $ L^1_{n,N}(\R^1) $ discrete-time martingales on \(\mathbb F^N\) and \(\mathbb F_n\), respectively.

\section{Problem formulation and reformulation}\label{sec:pb_formulate}

We are given running payoffs \((f^i)_{i\in\mathcal J}\), with \(f^i=(f^i(t))_{t\in[0,T]}\in\mathbb L^1(\mathbb R)\); terminal payoffs \((\Phi^i)_{i\in\mathcal J}\), with \(\Phi^i\in L^1(\mathcal F_T;\mathbb R)\); and \(\mathbb F\)-adapted switching costs \((l_{ij})_{i,j\in\mathcal J}\) satisfying for all $t\in [0,T]$, 

{(i)}. integrability $ l_{ij}(t) \in L^1(\mathcal{F}_t ; \R^1) $, $ i,j\in \J $;

{(ii)}. strict triangular condition
\(
l_{ii}(t)\equiv 0,\;
l_{ij}(t)+l_{jk}(t)>l_{ik}(t),\  i\neq j,\  j\neq k,\  \mathbb P\text{-a.s.}
\)

This rules out cost-improving instantaneous consecutive switches and makes the problem well posed.

\subsection{Original optimal switching problem}

For \(n\in\overline N\) and \(i\in\mathcal J\), an admissible switching control is a sequence
\(
\alpha=(\sigma_r,\kappa_r)_{r\ge0}\in\mathcal A_n^i
\)
such that \(\sigma_0=n\); \((\sigma_r)_{r\ge0} \in \calT_n \) ; \(\mathbb P(\sigma_r<N, 
\sigma_r=\sigma_{r+1})=0\) for all \(r\ge0\); \(\kappa_0=i\); and each \(\kappa_r\in\mathcal F_{t_{\sigma_r}}\) takes values in \(\mathcal J\), with
\(
\kappa_{r+1}\neq \kappa_r \; \text{on } \{\sigma_{r+1}<N\}.
\)
Define the number of effective switches by
\(
N(\alpha):=\sum_{r\ge1}1_{\{\sigma_r<N\}}.
\)
Since \(\sigma_r=N,\; \P\)-a.s.
for all sufficiently large \(r\), the reward
\begin{equation}\label{eq:payoff-functional}
J_n^i(\alpha)
:=\mathbb E_{t_n}\!\Big[
\sum_{r\ge0}\Big(
\int_{t_{\sigma_r}}^{t_{\sigma_{r+1}}} f^{\kappa_r}(s)\,ds
-l_{\kappa_r\kappa_{r+1}}(t_{\sigma_{r+1}})1_{\{\sigma_{r+1}<N\}}
\Big)
+\Phi^{\kappa_{N(\alpha)}}
\Big]
\end{equation}
is well defined. The corresponding value process (Snell Envelope) is
\begin{equation}\label{eq:original-primal}
\overline Y_{t_n}^i
:=\operatorname*{ess\,sup}_{\alpha\in\mathcal A_n^i}J_n^i(\alpha),
\quad i\in\mathcal J,\ \ n\in\overline N.
\tag{P0}
\end{equation}

\begin{remark}[Connection to QVIs]
In the Markovian case, \(\overline Y_{t_n}^i=\bar v_i^\pi(t_n,X_{t_n})\). If the regime-\(i\) dynamics have generator \(\mathcal A^i\) (possibly with a nonlocal jump term), and
\(
\mathcal R_i[w](t,x):=\max_{j\neq i}\{w_j(t,x)-l_{ij}(t,x)\},
\)
then the continuous-time values are characterized in viscosity sense by the coupled QVI / integro-QVI system: for $ i\in \J $, 
\[
\min\bigl\{
-\partial_t v_i(t,x)-\mathcal A^i v_i(t,x)-f^i(t,x),\
v_i(t,x)-\mathcal R_i[v](t,x)
\bigr\}=0,
\quad
v_i(T,x)=\Phi^i(x).
\]
On the grid \(\pi\), the discrete values satisfy
\[
\bar v_i^\pi(T,x)=\Phi^i(x),\quad
\bar v_i^\pi(t_n,x)
=
\max\big\{
\mathcal O_n^i[\bar v_i^\pi(t_{n+1},\cdot)](x),\
\mathcal R_i[\bar v^\pi](t_n,x)
\big\},
\]
where
\(
\mathcal O_n^i[\psi](x)
:=
\mathbb E\!\big[
\int_{t_n}^{t_{n+1}} f^i(s,X_s^{t_n,x,i})\,ds
+\psi(X_{t_{n+1}}^{t_n,x,i})
\big].
\)
Thus the present problem is the grid-restricted counterpart of the classical QVI/PIDE. We do not pursue mesh-refinement here; the dual constructions below do not rely on PDE representation.
\end{remark}

\subsection{Equivalent reformulation}\label{sec:primal_equiv}

Since \(l_{ii}\equiv0\), it is convenient to transform a control by the regime-decision on each interval \([t_m,t_{m+1})\). 

\vspace{0.3em}
\paragraph{Regime-decision Reformulation}
For \(n\in\overline N^{-1} \) and \(i\in\mathcal J\), let \(\mathcal D_n^i\) be the set of \(\mathcal J\)-valued sequences \(d=(d_m)_{m=n}^N\) such that \(d_m\) is \(\mathcal F_{t_m}\)-measurable for each \(m\), and \(d_N=d_{N-1}\); set \(\mathcal D_N^i:=\{(i)\}\). For \(d\in\mathcal D_n^i\), define
\[
L_n^i(d)
:=
\mathbb E_{t_n}\!\Big[
\sum_{m=n}^{N-1}\Big(
\int_{t_m}^{t_{m+1}} f^{d_m}(s)\,ds
-l_{d_m d_{m+1}}(t_{m+1})1_{\{m+1<N\}}
\Big)
-l_{i d_n}(t_n)
+\Phi^{d_N}
\Big].
\]

\begin{theorem}[Equivalence with regime-decision]\label{thm:switching-equiv-mode-decide-primal}
For any \(i\in\mathcal J\) and \(n\in\overline N\),
\begin{equation}\label{eq:original-primal-mode-choose}
\overline Y_{t_n}^i
=
\esssup_{j\in\mathcal D_n^i}L_n^i(j),
\quad \mathbb P\text{-a.s.} 
\tag{\textbf{P}}
\end{equation}
\end{theorem}

\vspace{0.3em}
\paragraph{Dual Upper Bound and Weak Duality.}
Given \(M=(M^j)_{j\in\mathcal J}\in(\mathcal M_{n,N})^J\), write
\(
\Delta M_{t_m}^j:=M_{t_{m+1}}^j-M_{t_m}^j.
\)
For \(i\in\mathcal J\), \(n\in\overline N\), and \(j=(j_m)_{m=n}^N\in\mathcal J_n\), define
\[
\widetilde U_n^{i,j}(M)
:=
\sum_{m=n}^{N-1}\Big(
\int_{t_m}^{t_{m+1}} f^{j_m}(s)ds
-l_{j_m j_{m+1}}(t_{m+1})1_{\{m+1<N\}}
-\Delta M_{t_m}^{j_m}
\Big)
-l_{i j_n}(t_n)
+\Phi^{j_N}
\]
and
\begin{equation}\label{eq:upper_bound_operator}
\widetilde U_n^i(M)
:=
\max_{j\in\mathcal J_n}\widetilde U_n^{i,j}(M).
\end{equation}
Equivalently, with the auxiliary index \(j_{n-1}:=i\),
\begin{equation}\label{eq:upper_bound_operator_equiv}
\widetilde U_n^i(M)
=
\max_{(j_m)_{m=n-1}^N\in\mathcal J_n^i}
\Big[
\sum_{m=n}^{N-1}\Big(
\int_{t_m}^{t_{m+1}} f^{j_m}(s)\,ds
-l_{j_{m-1}j_m}(t_m)
-\Delta M_{t_m}^{j_m}
\Big)
+\Phi^{j_N}
\Big].
\end{equation}
This operator is the basic dual upper-bound functional used below and in Section~\ref{sec:DeepSwitchMart}.

\begin{lemma}[Weak duality]\label{lem:weak-duality}
For any \(i\in\mathcal J\) and \(n\in\overline N\),
\begin{equation}\label{eq:mode-choosing-weak-dual}
\overline Y_{t_n}^i
\le
\operatorname*{ess\,inf}_{M\in(\mathcal M_{n,N})^J}
\mathbb E_{t_n}\!\big[\widetilde U_n^i(M)\big],
\qquad \mathbb P\text{-a.s.}
\end{equation}
\end{lemma}

\section{Duality of optimal switching problem: dual regime-decision}\label{sec:duality_equiv}

We first recall the duality for the associated classical iterated stopping problem (for detailed discussion of iterated stopping, see \Cref{sec:iterative-stopping-supplement} in the Supplementary Materials). 

For \(n\in\N^{-1}\), \(i,j\in\mathcal J\), set
\(
\mathcal R_n^{i,j}:=\overline Y_{t_n}^j-l_{ij}(t_n),
\;
\overline{\mathcal R}_n^i:=\max_{j\in\mathcal J^{-i}}\mathcal R_n^{i,j}.
\)

\subsection{Duality of iterative stopping problem}\label{subsec:dual_iterative_stopping}

By the Doob decomposition, there exist \(\overline M=(\overline M^i)_{i\in\mathcal J}\in(\mathcal M_N)^J\) and a family of nondecreasing \(\mathbb F^N\)-predictable processes \(\overline A=(\overline A^i)_{i\in\mathcal J}\), with \(\overline A_0^i=0\), such that
\[
\overline Y_{t_n}^i
=
\overline Y_0^i+\overline M_{t_n}^i-\overline A_{t_n}^i,
\qquad n\in\overline N.
\]
For \(i\in\mathcal J\), \(n\in\overline N\), \(M^i\in\mathcal M_{n,N}\), and \(m\in\overline N_n\), define
\begin{equation}\label{eq:iterative_dual_upper_two_regime}
    \overline U_{n,m}^i(M^i)
    :=
    \int_{t_n}^{t_m} f^i(s)\,ds
    +\overline{\mathcal R}_m^i\,1_{\{m<N\}}
    +\Phi^i\,1_{\{m=N\}}
    -M_{t_m}^i+M_{t_n}^i,
\end{equation}
and set
\(
\overline U_n^i(M^i):=\max_{m\in\overline N_n}\overline U_{n,m}^i(M^i).
\)
By \cite[Theorem~2.1]{Roger02} applied to the iterated stopping formulation, we obtain the following dual representation.

\begin{lemma}[Dual iterative stopping, surely optimal]\label{lem:dual_iterative}
For any \(i\in\mathcal J\), \(n\in\overline N\),
\begin{equation}\label{eq:equivalent-dual}
\overline Y_{t_n}^i
=
\operatorname*{ess\,inf}_{M^i\in\mathcal M_{n,N}}
\mathbb E_{t_n}\!\big[\overline U_n^i(M^i)\big]
=
\overline U_n^i(\overline M^i),
\qquad \mathbb P\text{-a.s.}   \tag{D0}
\end{equation}
\end{lemma}

Moreover, \Cref{lem:supermartingale_envelop}, together with \(l_{ii}\equiv0\), implies that for any \(i\in\mathcal J\) and any discrete stopping time \(\tau\in\mathcal T^N\),
\begin{equation}\label{eq:stopping-equality-regime-decision}
\overline Y_{t_\tau}^i
=
\overline{\mathcal R}_\tau^i\,1_{\{\tau<N\}}
+\Phi^i\,1_{\{\tau=N\}}.
\end{equation}

\begin{remark}[Incomputable upper bound]
    Although Lemma~\ref{lem:dual_iterative} yields a form of ``duality,'' the resulting upper bound is not computable, due to the coupling of the value processes $ \overline Y_{t_n}^j ,\; j\neq i $ in the reflection term $ \overline{\mathcal R}_n^i $ of the upper-bound operator $ \overline U_n^i $. This motivates us to develop a complete duality theory that yields a computable upper bound for $ \overline Y_{t_n}^i $.
\end{remark}
We exploit this sure optimality property to iteratively expand the inner value functions appearing in the dual upper bound, thereby deriving strong duality for the regime-decision formulation.

\subsection{Doob charaterization of martingale penalty}\label{subsec:surely_expansion}

By Lemma~\ref{lem:dual_iterative}, the Doob martingales \((\overline M^i)_{i\in\mathcal J}\) are surely optimal. For \(i\in\mathcal J\) and \(n\in\overline N\), define the induced stopping time and switching rule by
\begin{align}
\tau_n^i := \overline m(n,i)
& := \inf \big\{ \argmax_{m\in\overline N_n}\overline U_{n,m}^i(\overline M^i) \big\}, 
\label{eq:dual-stopping} \\
\iota_n^i:=j(n,i)
&:= \inf \big\{ \argmax_{j\in\mathcal J^{-i}}\mathcal R_n^{i,j}\big\} 1_{(n<N)} + i1_{(n=N)}.
\label{eq:mode-decision-distinguish-def}
\end{align}
The discussion of basic facts of \(\tau_n^i\) and \(\iota_n^i\), such as measurability, optimality and representation identity, are presented in Supplementary Materials.

We next introduce the events associated with possible consecutive switches.

\begin{definition}
For \(i\in\mathcal J\) and \(n\in\N^{-1}\), let
\begin{align*}
& A^{i,n} := \{\tau_n^{\iota_n^i}=n\}\cap\{\iota_n^{\iota_n^i}\neq i\}, \quad
B^{i,n} := \{\tau_{\tau_n^i}^{\,\iota_{\tau_n^i}^i}=\tau_n^i\},\\
& C^{i,n} := \{\tau_{\tau_n^i}^{\,\iota_{\tau_n^i}^i}=\tau_n^i\}\cap
           \{\iota_{\tau_n^i}^{\,\iota_{\tau_n^i}^i}=i\}\subset B^{i,n}, \quad
S^{i,n} := \{\tau_n^i<N\},\\
& D^{i,n} := \{\tau_n^i=n\}\cap\{\tau_n^{\iota_n^i}=n\} =(B^{i,n}\cap S^{i,n})\cap\{\tau_n^i=n\}\subset B^{i,n}\cap S^{i,n} .
\end{align*}
\end{definition}

Using the triangular condition on the switching costs, we rule out immediate consecutive switching.

\begin{lemma}[Sub-optimality of consecutive switches]\label{lem:no-free-lunch-dual}
For any \(i\in\mathcal J\) and \(n\in\N^{-1}\),
\(
\P(A^{i,n})=\P(B^{i,n}\cap S^{i,n})
=\P(C^{i,n}\cap S^{i,n})
=\P(D^{i,n})=0.
\)
Consequently,
\(
1_{S^{i,n}}=1_{(B^{i,n})^c\cap S^{i,n}}
\; \P\text{-a.s.}
\)

Set
\(
\widetilde A_m^{i,n}:=\{m<N\}\cap(A^{i,m})^c,
\) for \(m\in\overline N_n\). 
Then, 
\begin{align}
\overline Y_{t_n}^i
&=\max_{m\in\overline N_n}\Big(
\int_{t_n}^{t_m} f^i(s)\,ds
+\mathcal R_m^{i,\iota_m^i}1_{\widetilde A_m^{i,n}}
+\Phi^i1_{(m=N)}
-\overline M_{t_m}^i+\overline M_{t_n}^i
\Big), \label{eq:iterative-dual-surely-optimal-2}\\
\tau_n^i
&=\inf\;\argmax_{m\in\overline N_n}\Big(
\int_{t_n}^{t_m} f^i(s)\,ds
+\mathcal R_m^{i,\iota_m^i}1_{\widetilde A_m^{i,n}}
+\Phi^i1_{(m=N)}
-\overline M_{t_m}^i+\overline M_{t_n}^i
\Big), \label{eq:iterative-dual-surely-optimal-stopping-time-2}\\
\overline Y_{t_n}^i
&=\int_{t_n}^{t_{\tau_n^i}} f^i(s)\,ds
+\mathcal R_{\tau_n^i}^{i,\iota_{\tau_n^i}^i}\,
1_{(B^{i,n})^c\cap S^{i,n}}
+\Phi^i1_{(\tau_n^i=N)}
-\overline M_{t_{\tau_n^i}}^i+\overline M_{t_n}^i.
\label{eq:stopped-Y}
\end{align}
\end{lemma}


We next construct admissible regime-decision candidates, which would satisfy the required maximization properties in subsequent sections.

\begin{theorem}[Optimal regime-decision candidate]\label{thm:construction-optimal-mode-decision}
Define \(j^{i,m}=(j_k^{i,m})_{k=m}^N\) backwardly:
\(
j_N^{i,N}:=\iota_N^i=i,\; i\in\mathcal J
\)
, and for \(m=N-1,\dots,0\), \(k\in\overline N_m\), \( i\in \J \), with notation \(\hat\tau_m^i:=\tau_m^{\iota_m^i}\), \( T^{i,1}_m:= \{\tau_m^i>m\} ,\; T^{i,2}_m := \{\tau_m^i=m\} \), define 
\begin{equation}\label{eq:construction-mode-decision}
j_k^{i,m}:=
\begin{cases}
i, & m\le k<\tau_m^i,\\
j_k^{\iota_{\tau_m^i}^i,\tau_m^i}, & \tau_m^i\le k\le N,
\end{cases}
 \text{on } \; T^{i,1}_m, \; \text{and}  \; \; \; \\
\begin{cases}
\iota_m^i, & m\le k<\hat\tau_m^i,\\
j_k^{\iota_{\hat\tau_m^i}^{\iota_m^i},\,\hat\tau_m^i}, & \hat\tau_m^i\le k\le N,
\end{cases}
 \text{on } \; T^{i,2}_m.
\end{equation}
Then, for every \(n\in\overline N\) and \(i\in\mathcal J\), \(j^{i,n}\) is well defined, \(\mathbb F_n\)-adapted, and \(j^{i,n}\in\mathcal D_n^i\). Moreover, \(\mathbb P\)-a.s.:

\emph{(i)}. 
        \(j_n^{i,n}\in\argmax_{j\in\mathcal J}[\mathcal R_n^{i,j}1_{(n<N)}] \). If $n<N$,
        \(
        j_k^{i,n}\in\argmax_{j\in\mathcal J} [ \mathcal R_k^{j_{k-1}^{i,n},j}1_{(k<N)} ] \), 
        \( k\in\overline N_{n+1} \), and furthermore,  
    \begin{itemize}
        \item for \(k\in\N_{n+1}^{-1}\), \( \overline Y_{t_k}^{j_{k-1}^{i,n}} >\overline{\mathcal R}_k^{j_{k-1}^{i,n}}  \) if \( \tau_k^{j_{k-1}^{i,n}}>k \), and \( \overline Y_{t_k}^{j_{k-1}^{i,n}}  = \overline{\mathcal R}_k^{j_{k-1}^{i,n}} \) if \( \tau_k^{j_{k-1}^{i,n}}=k \); 
        \item  \( \overline Y_{t_n}^{i} >\overline{\mathcal R}_n^{i} \) if \( \tau_n^i>n \), and \( \overline Y_{t_n}^{i} =\overline{\mathcal R}_n^{i} \), \( Y_{t_n}^{\iota_n^i}>\overline{\mathcal R}_n^{\iota_n^i} \) if \( \tau_n^i=n \). 
    \end{itemize}


\emph{(ii)}. (DPP) If \(n<N\), then
\(
j_k^{i,n}=j_k^{j_n^{i,n},\,n+1},\; k\in\overline N_{n+1} ;
\)

\emph{(iii)}. 
\(
j_k^{i,n}=j_k^{\iota_n^i,n},\; k\in\overline N_n
\) if \( \tau_n^i=n \). 
\end{theorem}

We now derive the pathwise expansion of \(\overline Y\) along the candidates \(j^{i,n}\).   

\begin{theorem}[Surely expansion theorem]\label{thm:expansion-surely-optimal}
For the candidates \(j^{i,n}\) in \eqref{eq:construction-mode-decision},
\(
\overline Y_{t_n}^i=\widetilde U_n^{i,j^{i,n}}(\overline M),
\; i\in\mathcal J,\ n\in\overline N,\ \mathbb P\text{-a.s.}
\)
\end{theorem}

\subsection{Dual dynamic programming principle}\label{subsec:Dual_DPP}

In this section we establish the dynamic programming principle for the dual upper bound \eqref{eq:upper_bound_operator}, which is the key ingredient for strong duality.

For fixed \(n\in\overline N\), define pathwise
\(
\mathcal J_S(\omega):=\{j\in\mathcal J: j_n^{j,n}(\omega)=j\} \subset \J ,
\;
\mathcal J_N(\omega):=\mathcal J\setminus \mathcal J_S(\omega).
\)
The next lemma shows that the candidate rule \(j^{i,n}\) cannot switch twice at time \(t_n\) with positive probability.

\begin{lemma}\label{lem:switching-times-no-free-lunch}
Fix \(n\in\overline N\). Then, \(\mathbb P\)-a.s.,
\[
    \emph{(i)}. \; \; \mathcal J_S\neq\varnothing ; \; \; \emph{(ii)}. \; \; j_n^{j,n}\in\mathcal J_S ,\; \forall \; j\in\mathcal J ,\; \text{or equivalently}, \;  j_n^{\,j_n^{j,n},\,n}=j_n^{j,n},\; j\in\mathcal J.
\]
\end{lemma}

\begin{corollary}\label{cor:max-no-free-lunch}
For any \(i\in\mathcal J\), \(
\max_{j\in\mathcal J}\mathcal R_n^{i,j}
=
\max_{j\in\mathcal J_S}\mathcal R_n^{i,j} \; \; \mathbb P\text{-a.s.}
\)  for $n\in \N^{-1} $ and 
\(
\overline Y_{t_n}^i
=
\max_{j\in\mathcal J_S}\mathcal R_n^{i,j}\,1_{(n<N)}
+\Phi^i\,1_{(n=N)} \;
\; \mathbb P\text{-a.s.}
\) for $n\in \N$.
\end{corollary}

\begin{remark}
Thus optimal regime decisions at time \(t_n\) may be restricted to \(\mathcal J_S\), i.e., to non-consecutive switches.
\end{remark}

To prepare for strong duality, we also write the surely expansion from \Cref{thm:expansion-surely-optimal} explicitly. Setting \(j_{n-1}^{i,n}:=i\), we have, for \(i\in\mathcal J\) and \(n\in\overline N\),
\[
\overline Y_{t_n}^i
=
\sum_{k=n}^{N-1}
\Big(
\int_{t_k}^{t_{k+1}} f^{j_k^{i,n}}(s)\,ds
-
l_{j_{k-1}^{i,n}j_k^{i,n}}(t_k)
-
\Delta\overline M_{t_k}^{\,j_k^{i,n}}
\Big)
+\Phi^{j_N^{i,n}},
\quad \mathbb P\text{-a.s.}
\]

\begin{theorem}[Dual dynamic programming principle]\label{thm:DPP-upper-surely-optimal}
For any \(i\in\mathcal J\), \(n\in\N^{-1}\), and \(M=(M^j)_{j\in\mathcal J}\in(\mathcal M_{n,N})^J\), we have
\(
\widetilde U_N^i(M)=\Phi^i=\overline Y_{t_N}^i,
\)
and, \(\mathbb P\)-a.s.,
\begin{align}
\overline Y_{t_n}^i
&=
\int_{t_n}^{t_{n+1}} f^{j_n^{i,n}}(s)\,ds
-l_{i\,j_n^{i,n}}(t_n)
-\Delta\overline M_{t_n}^{\,j_n^{i,n}}
+\overline Y_{t_{n+1}}^{\,j_n^{i,n}}
\notag\\
&=
\max_{j\in\mathcal J}
\Big[
\int_{t_n}^{t_{n+1}} f^j(s)\,ds
-l_{ij}(t_n)
-\Delta\overline M_{t_n}^{\,j}
+\overline Y_{t_{n+1}}^{\,j}
\Big], \label{eq:first-second-equality-Y}\\
\widetilde U_n^i(M)
&=
\max_{j\in\mathcal J}
\Big[
\int_{t_n}^{t_{n+1}} f^j(s)\,ds
-l_{ij}(t_n)
-\Delta M_{t_n}^{\,j}
+\widetilde U_{n+1}^j(M)
\Big].
\label{eq:equality-dual-upper-bound}
\end{align}
\end{theorem}

For subsequent convergence analysis, we provide the following error propogation lemma. Since the proof is similar to \cite{ye2025deepmartingale}, we omit here.
\begin{lemma}[Error Propagation] \label{lem:error-propagate}
    Define the martingale difference operators by $ \xi^i_n : (\mathcal{M}_{n,N})^{J} \ni (M^i)_{i\in \mathcal{J}} \mapsto \Delta M^{i}_{t_n} $, for $ i\in \mathcal{J} $ and $ n \in \N^{-1} $. Then, for any $ M^{\circ},M^{\star} \in (\mathcal{M}_{n,N})^{J} $,
    \begin{equation}\label{eq:error_propogation}
        \begin{aligned}
            \max_{i\in \mathcal{J}} \big|\widetilde{U}^i_n(M^{\circ}) -\widetilde{U}^i_n(M^{\star}) \big|  & \le \max_{i\in \mathcal{J}} \big|\widetilde{U}^i_{n+1}(M^{\circ}) - \widetilde{U}^i_{n+1}(M^{\star}) \big| + \max_{i\in \mathcal{J}} |\xi^i_{n}(M^{\circ}) - \xi^i_{n}(M^{\star}) | \\
                & \le \max_{i\in \mathcal{J}} \big|\widetilde{U}^i_{n+1}(M^{\circ}) - \widetilde{U}^i_{n+1}(M^{\star}) \big| + \sum_{i \in \mathcal{J}} |\xi^i_{n}(M^{\circ}) - \xi^i_{n}(M^{\star}) | .
        \end{aligned}
    \end{equation}
\end{lemma}

\subsection{Strong duality and computable upper bound}\label{subsec:strong_dual}

We next state strong duality for the regime-decision formulation; in particular, the Doob martingales \(\overline M\) are minimal martingale penalties.

\begin{theorem}[Strong duality, surely optimal]\label{thm:stong_duality}
For any \(i\in\mathcal J\) and \(n\in\overline N\),
\begin{equation}\label{eq:surely-optimal-strong-dual}
\overline Y_{t_n}^i
=
\widetilde U_n^{i,j^{i,n}}(\overline M)
=
\widetilde U_n^i(\overline M)
=
\mathbb E_{t_n}\!\big[\widetilde U_n^i(\overline M)\big]
=
\essinf_{M\in(\mathcal M_{n,N})^J}\mathbb E_{t_n}\!\big[\widetilde U_n^i(M)\big],
\; \; \mathbb P\text{-a.s.}  \tag{\textbf{D}}
\end{equation}
\end{theorem}

\subsection{Primal dynamic programming principle and auxiliary lower bound construction}

Before introducing \textit{DeepMartingales}, we record the primal dynamic programming principle and the optimality of the candidates $j^{i,n}$, both of which will be used in the expressivity analysis and in the numerical section, to construct feasible lower bounds for comparison.

\begin{proposition}[Primal dynamic programing principle and optimality]\label{pro:primal_DPP}
For \(i\in\mathcal J\), \(n\in\N^{-1}\), \(d^n=(d_m)_{m=n}^N \in \mathcal D_n^i\), write
\(
d^{n+1}:=(d_m)_{m=n+1}^N  .
\)
Then \(L_N^i(\cdot)\equiv \Phi^i\), 
\begin{align}
L_n^i(d^n)
&=\E_{t_n}\!\Big[\int_{t_n}^{t_{n+1}} f^{d_n}(s)\,ds + L_{n+1}^{d_n}(d^{n+1}) - l_{i d_n}(t_n)\Big], \label{eq:lower_bound_DPP}\\
\overline Y_{t_n}^i
&=\max_{j\in\mathcal J}\E_{t_n}\!\Big[\int_{t_n}^{t_{n+1}} f^j(s)\,ds + \overline Y_{t_{n+1}}^j - l_{ij}(t_n)\Big], \quad \P\text{-a.s.} \label{eq:primal_DPP_value_process}
\end{align}
Moreover,
\begin{equation}\label{eq:primal_optimality}
\overline Y_{t_n}^i=L_n^i(j^{i,n}), \quad i \in \mathcal J,\  n\in\overline N,\  \P\text{-a.s.}
\end{equation}

\end{proposition}

\begin{remark}
    By verfication, \eqref{eq:lower_bound_DPP} is equivalent to
    \begin{equation}\label{eq:primal_DPP_indicator}
        L_n^i(d^n)
        =\E_{t_n}\!\Big[\sum_{j=1}^J 1_{\{d_n=j\}}
        \Big(\int_{t_n}^{t_{n+1}} f^j(s)\,ds + L_{n+1}^j(d^{n+1}) - l_{ij}(t_n)\Big)\Big].
    \end{equation}

\end{remark}

\section{DeepMartingale solver}\label{sec:DeepSwitchMart}

We adapt \textit{DeepMartingale} \cite{ye2025deepmartingale} to the dual switching problem in Brownian Markovian setting. Let \((\Omega,\mathcal F,\mathbb F,\mathbb P)\) support a \(d\)-dimensional Brownian motion \(W=(W^1,\ldots,W^d)^\top\), and let \(\mathbb F\) be its augmented filtration. We consider $X$  the unique strong solution of the following It\^o diffusion
\begin{equation}\label{eq:SDE}
dX_t=\mu(t,X_t)\,dt+\sigma(t,X_t)\,dW_t,\quad X_0=x,
\end{equation}
where \(\mu:[0,T]\times\mathbb R^d\to\mathbb R^d\) and \(\sigma:[0,T]\times\mathbb R^d\to\mathbb R^{d\times d}\) are Lipschitz in \(x\) and \(1/2\)-H\"older in \(t\). Regime-dependent dynamics can be handled by state augmentation, so we restrict to a common state process \(X\).

We consider the following \textbf{Markovian structure}: for \(i,j\in\mathcal J\), the functions
\(
f^i:[0,T]\times\mathbb R^d\to\mathbb R,\;
\Phi^i:\mathbb R^d\to\mathbb R,\;
l_{ij}:[0,T]\times\mathbb R^d\to\mathbb R
\)
are Borel measurable, satisfy standard polynomial-growth conditions, and satisfy
\(
l_{ii}(t,x)\equiv0,\;
l_{ij}(t,x)+l_{jk}(t,x)>l_{ik}(t,x),
\;
t\in[0,T],\ \text{a.e. }x,\  i\neq j,\ j\neq k.
\)
Then the switching values are well defined, admit the Markovian representation
\(
\overline Y_{t_n}^i=V_n^i(X_{t_n}),\; i\in\mathcal J, \ n\in\overline N,
\)
for measurable \(V_n^i:\mathbb R^d\to\mathbb R\), and satisfy \(\overline Y_{t_n}^i\in L^2(\mathcal F_{t_n};\mathbb R)\); see \cite[Lemma~B.1]{ye2025deepmartingale}.

\vspace{0.3em}
\paragraph{Martingale Discretization}
By the martingale representation theorem, the Doob martingales \(\overline M=(\overline M^i)_{i\in\mathcal J}\) satisfy
\(
\overline M_{t_n}^i=\int_0^{t_n}\overline Z_s^i\cdot dW_s,
\;
\overline Z^i\in\mathbb L^2(\mathbb R^d).
\)
Following \cite{belome09,ye2025deepmartingale}, partition each \([t_n,t_{n+1}]\), \(n\in\N^{-1}\), into \(K \in \mathbb{N}_{+} \) uniform subgrid
\(
t_n=t_0^n<\cdots<t_K^n=t_{n+1},\;
\Delta t:=t_{k+1}^n-t_k^n= {T}/{NK},
\)
and write \(\Delta W_{t_k^n}:=W_{t_{k+1}^n}-W_{t_k^n}\). Define
\begin{equation}\label{eq_Z}
\hat Z_{t_k^n}^{i;K}
:=\frac1{\Delta t}\E_{t_k^n}\!\big[\overline Y_{t_{n+1}}^i  \Delta W_{t_k^n}\big],
\quad
\hat M_{t_n}^{i;K}
:=\sum_{m=0}^{n-1}\sum_{k=0}^{K-1}\hat Z_{t_k^m}^{i;K}\cdot\Delta W_{t_k^m},
\end{equation}
for \(k\in \K := \{0,\dots,K-1\}\), \(n\in\overline N^{-1} \). Then \(\hat M^K:=(\hat M^{i;K})_{i\in\mathcal J}\in(\mathcal M_N)^J\).

\vspace{0.3em}
\paragraph{Pure Dual Backward Minimization}
Let
\(
\mathcal P_n:=\{\xi\in\mathcal F_{t_{n+1}}: \E_{t_n}[\xi]=0\}.
\)
For \(\tilde\xi_n=(\xi_n^i)_{i\in\mathcal J}\in(\mathcal P_n)^J\) and \(\widetilde M^{n+1}\in(\mathcal M_{n+1,N})^J\), let
\[
\widetilde M^n:= \tilde\xi_n + \widetilde M^{n+1} \in(\mathcal M_{n,N})^J,
\quad
\widetilde U_n^i(\tilde\xi_n;\widetilde M^{n+1}):=\widetilde U_n^i(\widetilde M^n) .
\]
If \(\eta_n^i\in\mathcal F_{t_n}\) satisfies \(\eta_n^i\le \overline Y_{t_n}^i \; \; \P \)-a.s., then for any \(M\in(\mathcal M_{n,N})^J\),
\(
\E\big|\widetilde U_n^i(M)-\eta_n^i\big|^2
\ge
\E\big|\overline Y_{t_n}^i-\eta_n^i\big|^2,
\)
with equality at \(M=\overline M\). Motivated by \Cref{thm:DPP-upper-surely-optimal}, we therefore solve the dual problem backwardly by minimizing either the dual upper bound itself or its \(L^2\)-surrogate.

\begin{pr}[Pure dual backward minimization]\label{pb:backward_minimization}
Fix \(n\in\N^{-1}\), \(i\in\mathcal J\), and suppose \(\tilde\xi_{n+1},\ldots,\tilde\xi_{N-1}\) have been determined. Choose \(\tilde\xi_n\in(\mathcal P_n)^J\) by solving
\begin{align}
\tilde\xi_n
&\in \arginf_{\xi_n\in(\mathcal P_n)^J}
\E\big[\widetilde U_n^i(\xi_n;\widetilde M^{n+1})\big],
&&\text{(upper-bound loss)}, \tag{D1}\label{eq_optim}\\
\tilde\xi_n
&\in \arginf_{\xi_n\in(\mathcal P_n)^J}
\E\big|\widetilde U_n^i(\xi_n;\widetilde M^{n+1})-\eta_n^i\big|^2,
&& \text{(\(L^2\)-surrogate loss)}. \tag{D2}\label{eq_optim_second_moment}
\end{align}
\end{pr}

\begin{remark}\label{remark:pure_dual_minimization}
The Doob martingales \(\overline M\) solve both \eqref{eq_optim} and \eqref{eq_optim_second_moment} for every \(i\in\mathcal J\). In practice, \eqref{eq_optim} is typically slightly tighter but less stable, whereas \eqref{eq_optim_second_moment} is more stable and also performs better for the hedging application in Section~\ref{subsec:delta_deephedge}. The lower bound \(\eta_n^i\) may be obtained analytically when simple lower bounds for model primitives \(f^i\), \(-l_{ij}\), and \(\Phi^i\) are available, or tuned as a hyperparameter. For some simple cases, $ \eta^i_n \equiv 0 $ is the most convenient choice.  
\end{remark}

\vspace{0.3em}
\paragraph{DeepMartingale Parametrization}
Let \(\Theta:=\bigcup_{m\ge1}\mathbb R^m\). For each \(n\in\N^{-1}\) and \(i\in\mathcal J\), let \(z_n^{i,\theta_n^i}:\mathbb R^{1+d}\to\mathbb R^d\) be a feedforward neural network with parameter \(\theta_n^i\in\Theta\),
\[
z_n^{i,\theta_n^i}
=
a_{I+1}^{\theta_n^i}\circ \varphi_{q_I}\circ a_I^{\theta_n^i}\circ \cdots \circ \varphi_{q_1}\circ a_1^{\theta_n^i},
\]
where \(I\ge1\), \(q_1,\dots,q_I\in\mathbb N_+\), the \(a_\ell^{\theta_n^i}\) are affine maps, and \(\varphi_q\) denotes the component-wise bounded, non-constant activation \(\varphi\). Define the \textit{DeepMartingales} \(M^{\theta;K}=(M^{i,\theta^i;K})_{i\in\mathcal J}\) by
\begin{equation}\label{deep_mart}
\xi_n^{i,\theta_n^i;K}
:=\sum_{k=0}^{K-1} z_n^{i,\theta_n^i}(t_k^n,X_{t_k^n})\cdot \Delta W_{t_k^n}\,1_{\{n<N\}},
\; \; \;
M_{t_n}^{i,\theta^i;K}
:=\sum_{m=0}^{n-1}\xi_m^{i,\theta_m^i;K},
\; 
n\in\overline N,
\end{equation}
where \(\theta^i:=(\theta_n^i)_{n=0}^{N-1}\in\Theta^N\) and \(\theta:=(\theta^i)_{i\in\mathcal J}\in\Theta^{N\times J}\). Note that \(\xi_n^{i,\theta_n^i;K}\in\mathcal P_n\).

\subsection{Convergence under bounded activation function}

The next two convergence result follows from \cite[Theorems~4.6--4.7]{ye2025deepmartingale} together with \Cref{lem:error-propagate}, and thus we omit the proofs.

\begin{theorem}\label{thm:theta_approx}
For any \(\varepsilon>0\), there exists a family of DeepMartingales \(M^{\theta_\varepsilon;K}\) such that, for each \(n\in\overline N\),
\(
\E\big[\max_{i\in\mathcal J}
\big|\widetilde U_n^i(\hat M^K)-\widetilde U_n^i(M^{\theta_\varepsilon;K})\big|^2\big]
\le (N-n)J\varepsilon.
\)
\end{theorem}

Hence the deep upper bounds are asymptotically tight.

\begin{corollary}\label{coro:tight_upper}
For any \(n\in\overline N\), \(i\in\mathcal J\),

\emph{(i)}.
\(
\E[\overline Y_{t_n}^i]
=
\lim_{K\to\infty}\inf_{\theta\in\Theta^{N\times J}}
\E\big[\widetilde U_n^i(M^{\theta;K})\big] ;
\)

\emph{(ii)}.
for any \(\eta_n^i\in\mathcal F_{t_n}\),
\(
\E\big|\overline Y_{t_n}^i-\eta_n^i\big|^2
=
\lim_{K\to\infty}\inf_{\theta\in\Theta^{N\times J}}
\E\big|\widetilde U_n^i(M^{\theta;K})-\eta_n^i\big|^2.
\)
\end{corollary}

The next proposition, in the spirit of \cite[Proposition~4.9]{ye2025deepmartingale}, justifies the $L^2$-surrogate loss \eqref{eq_optim_second_moment}, which derives a converged and stable dual upper bounds.

\begin{proposition}\label{pro:L2-to-L1-loss}
Fix \(i\in\mathcal J\), \(n\in\overline N\). Assume \(\eta_n^i\in\mathcal F_{t_n}\) and
\(
\eta_n^i\le \overline Y_{t_n}^i,\; \P\text{-a.s.}
\)
Let \(\varepsilon_K\downarrow0\), and for each \(K\ge1\) choose \(\theta_i^K\in\Theta^{N\times J}\) such that
\[
\E\big|\widetilde U_n^i(M^{\theta_i^K;K})-\eta_n^i\big|^2
\le
\inf_{\theta\in\Theta^{N\times J}}
\E\big|\widetilde U_n^i(M^{\theta;K})-\eta_n^i\big|^2+\varepsilon_K.
\]
Then, as $ K \uparrow \infty $, 
\begin{itemize}
    \item[\emph{(i)}.]
    \(
    \E\!\big[\Var_n\!\big(\widetilde U_n^i(M^{\theta_i^K;K})\big)\big]\to 0
    \);
    \item[\emph{(ii)}.]
    \(
    \widetilde U_n^i(M^{\theta_i^K;K}) \xrightarrow{L^2} \overline Y_{t_n}^i
    \),
    and hence
    \(
    \E\big[\widetilde U_n^i(M^{\theta_i^K;K})\big]\to \E[\overline Y_{t_n}^i].
    \)
\end{itemize}
\end{proposition}

\subsection{Convergence \& Expressivity under ReLU activation}\label{subsec:converge_express_ReLU}

We now study the expressivity of \textit{DeepMartingale} under ReLU activations; see \cite{ye2025deepmartingale,gonon23,Jentzen23}. Throughout, we assume \(d\ge 3\) (no essential difference and only for simplicity), and notate the dimension dependence by super/sub-script. Let \(X^{t_n,x;d}\) denote the diffusion \eqref{eq:SDE} in dimension \(d\), started from \(x\in\mathbb R^d\) at time \(t_n\). By \Cref{pro:primal_DPP}, \(V_N^{i;d}=\Phi^{i;d}\), and
\begin{equation}\label{eq:primal_DPP_value_function}
V_n^{i;d}(x)
=
\max_{j\in\mathcal J}
\E\!\Big[
\int_{t_n}^{t_{n+1}} f^{j;d}(s,X_s^{t_n,x;d})\,ds
+V_{n+1}^{j;d}(X_{t_{n+1}}^{t_n,x;d})
-l_{ij}^d(t_n,x)
\Big] .
\end{equation}

On each interval \([t_n,t_{n+1}]\), martingale representation of \(\overline M^d\) amounts to solving the non-driver decoupled FBSDEs
\begin{equation}\label{eq:BSDE_system}
\begin{aligned}
X_t^{t_n,x;d}
&=x+\int_{t_n}^t \mu^d(s,X_s^{t_n,x;d})\,ds+\int_{t_n}^t \sigma^d(s,X_s^{t_n,x;d})\,dW_s^d,\\
\widetilde Y_t^{i,t_n,x;d}
&=V_{n+1}^{i;d}(X_{t_{n+1}}^{t_n,x;d})-\int_t^{t_{n+1}}\overline Z_s^{i,t_n,x;d}\,dW_s^d,
\quad i\in\mathcal J.
\end{aligned}
\end{equation}

\vspace{0.5em}
\paragraph{Numerical Integration Expressivity}

For a map \(P\) in the space variable, let \(\Lip P\) denote its minimal Lipschitz constant; for a matrix-valued \(P=(P^1,\ldots,P^d)\), set
\(
\LipH P:=\big(\sum_{r=1}^d |\Lip P^r|^2\big)^{1/2};
\)
and for a time-dependent map \(Q\), let \(\Holdertwo Q\) denote its minimal \(1/2\)-H\"older constant in time. If $Q$ is function of $(t,x)$, denote \(\Holdertwo Q := \sup_{x\in \R^d } \frac{\Holdertwo Q(\cdot,x) }{1+\|x\|} \). Norms are Euclidean or Hilbert--Schmidt, as appropriate.

\begin{as}\label{ass:N_0_structural_ass_model_deter}
There exist constants \( c,q>0 \) independent of \( d \), such that the functions
\( F_1^{d}(t,x) \in \{ \mu^{d}(t,x),\sigma^{d}(t,x) \} \),
\( F_2^{d}(t,x) \in \{ f^{i;d}(t,x) :\; i\in \J \} \),
and
\( G^{d}(x) \in \{ \Phi^{i;d}(x), l_{ij}^{d}(t_n,x) : \; i,j\in \J, n\in \N^{-1}   \} \), 
satisfy
for any \( t \in [0,T] \), \( x \in \mathbb{R}^{d} \):
\begin{itemize}
    \item[(i).]  \( \Lip \mu^{d}(t,\cdot) \le c (\log d)^{\frac{1}{2}} \) and \( \LipH \sigma^{d}(t,\cdot) \le c (\log d)^{\frac{1}{4}} \);
    \item[(ii).]  \( \Holdertwo F_1^{d}(\cdot,x) \), \( \Lip F_2^{d}(t,\cdot) \), \( \|F_1^{d}(t,\mathbf{0})\| \), \( \|F_2^{d}(t,\mathbf{0})\| \), \( \Lip G^{d} \) and \( \|G^{d}(\mathbf{0})\| \) are all bounded by \( c d^{q} \).
\end{itemize}
\end{as}

To apply \cite[Theorem~3.9]{ye2025deepmartingale} on each interval, we need the following inheritance property, the analogue of \cite[Proposition~A.16]{ye2025deepmartingale}.

\begin{lemma}\label{lem:V_Lip_growth}
Under Assumption~\ref{ass:N_0_structural_ass_model_deter}, for each \(i\in\mathcal J\) and \(n\in\overline N_1\), the value function \(V_n^{i;d}\) also satisfies Assumption~\ref{ass:N_0_structural_ass_model_deter} (ii) with \( G^{d}(x) = V_n^{i;d}(x) \).

\end{lemma}


Since Assumption~\ref{ass:N_0_structural_ass_model_deter} on the It\^o diffusion coefficients implies the conditions required in \cite[Theorem~3.9]{ye2025deepmartingale}, Lemma~\ref{lem:V_Lip_growth} together with the same argument as in the proof of \cite[Theorem~3.10]{ye2025deepmartingale} (the procedure is identical, and thus the proof is omitted) yields the following expressivity result. In particular, by choosing a sufficiently fine nested integration grid, 
one can attain an arbitrary approximation accuracy.

\begin{theorem}[Expressivity of numerical integration]\label{thm:express_N_0_new}
Under Assumption~\ref{ass:N_0_structural_ass_model_deter}, there exist constants \(b^*,q^*>0\), independent of \(d\), such that for any \(\varepsilon>0\) there exists \(K_{\varepsilon;d}\in\mathbb N_{+}\) satisfies
\(
K_{\varepsilon;d}\le b^* d^{q^*}\varepsilon^{-1} ,
\)
so that for all \(n\in\N^{-1}\), \(i\in\mathcal J\),
\[
\E\!\Big[
\sum_{k=0}^{K_{\varepsilon;d}-1}\int_{t_k^n}^{t_{k+1}^n}
\big\|\overline Z_s^{i;d}-\hat Z_{t_k^n}^{i;K_{\varepsilon;d}}\big\|^2\,ds
\Big]
\le \varepsilon .
\]
\end{theorem}

\vspace{0.5em}
\paragraph{Expressivity of DeepMartingale in the Optimal Switching Problem}
As in \cite{gonon23,ye2025deepmartingale}, we impose structural assumptions on the stochastic flow and on the reward/cost functions. 
For a map \(\psi:\mathbb R^d\to\mathbb R^m\), let
\(
\Growth(\psi):=\sup_{x\in\mathbb R^d}\frac{\|\psi(x)\|}{1+\|x\|},
\)
and let \(\size(\cdot)\) denote the number of nonzero entries of network parameters (see \cite[Section~4.3.1]{ye2025deepmartingale}).


Since $ X^d $ is the unique strong solution of It\^o diffusion \eqref{eq:SDE}, according to \cite[Proof of Theorem 7.1.2]{SDE03}, there exists a map \( P^d \) such that
\( X_t^{s,x;d}(\omega) = P^d(x,s,t,\omega) \) for \( s \le t \le T \),
where \( (x,s,t) \mapsto P^d(x,s,t,\omega) \) is \( \mathcal{B}(\mathbb{R}^{d+2}) \)-measurable. Define the stochastic flow \( P_s^{t;d}(x,\omega) := P^d(x,s,t,\omega) \) for any \( 0 \le s < t \le T \). 

\begin{as}[Stochastic flow assumption with order $p$]\label{ass:dynamic_ass}
    There exist constants \( c,q >0 \) independent of $d$, such that for any \( n \in \N^{-1} \), \( t_n \le s < t \le t_{n+1} \), and \( x \in \mathbb{R}^d \), the following properties hold:
    \begin{enumerate}
        \item[(a)] \( \| \Growth(P_s^{t;d}(*_x,\cdot_\omega)) \|_{L^p} \le c d^q \), and
        \( \frac{\mathbb{E}\| P_{t_n}^{t;d}(x,\cdot) - P_{t_n}^{s;d}(x,\cdot) \|}{(1+\|x\|)\sqrt{|t-s|}} \le c d^q \);
        
        \item[(b)] 
        there exists a RanNN (see \cite[Definition~4.13]{ye2025deepmartingale}) \( \hat{P}_s^{t;d} : \mathbb{R}^d \times \Omega \to \mathbb{R}^d \) with depth \( I_s^{t;d} \le c d^q \) such that \( P_s^{t;d} \) admits a realization \( \hat{P}_s^{t;d} \), i.e.,
        \( P_s^{t;d}(x,\omega) = \hat{P}_s^{t;d}(x,\omega) \) for all \( x \in \mathbb{R}^d \), \( \mathbb{P} \)-a.s. \( \omega \);
        
        \item[(c)] the RanNN realization \( \hat{P}_s^{t;d} \) in (b) satisfies \( \mathbb{E}[\size(\hat{P}_s^{t;d}(*,\cdot))] \le c d^q \).
\end{enumerate}
\end{as}

\begin{as}\label{ass:other_new}
There exist \(c,q,r>0\), independent of \(d\), such that for any \(\varepsilon\in(0,1]\), \(i,j\in\mathcal J\), and \(n\in\N^{-1}\), there exist deep ReLU networks
\(
\hat f_{\varepsilon}^{i;d}:[0,T]\times\mathbb R^d\to\mathbb R,\;
\hat\Phi_{\varepsilon}^{i;d}:\mathbb R^d\to\mathbb R,\;
\hat l_{ij,\varepsilon}^{n;d}:\mathbb R^d\to\mathbb R,
\)
such that for all \(t\in[0,T]\) and \(x\in\mathbb R^d\),
\[
|\hat f_{\varepsilon}^{i;d}(t,x)-f^{i;d}(t,x)|
+|\hat\Phi_{\varepsilon}^{i;d}(x)-\Phi^{i;d}(x)|
+|\hat l_{ij,\varepsilon}^{n;d}(x)-l_{ij}^d(t_n,x)|
\le \varepsilon cd^q(1+\|x\|),
\]
\[
\max\Big\{
\size(\hat f_{\varepsilon}^{i;d}),
\size(\hat\Phi_{\varepsilon}^{i;d}),
\size(\hat l_{ij,\varepsilon}^{n;d}),
\Growth(\hat f_{\varepsilon}^{i;d}(t,\cdot)),
\Growth(\hat\Phi_{\varepsilon}^{i;d}),
\Growth(\hat l_{ij,\varepsilon}^{n;d})
\Big\}
\le cd^q\varepsilon^{-r},
\]
and
\(
\Lip(\hat f_{\varepsilon}^{i;d}(t,\cdot))
+ \Holdertwo \hat f_{\varepsilon}^{i;d} 
\le cd^q .
\)
\end{as}

\begin{remark}\label{rem:reward_ass_express_to_ass_integration}
After enlarging \(c,q\) if necessary, the same constants may be used in Assumptions~\ref{ass:dynamic_ass} and \ref{ass:other_new}. Moreover, Assumption~\ref{ass:other_new} implies \( f^{i;d}, \Phi^{i;d}, l_{ij}^{d} \), \( i,j \in \mathcal{J} \) satisfy Assumption~\ref{ass:N_0_structural_ass_model_deter} by adapting the proof procedure of \Cref{lem:Lip_Hol_lem_f_i}. 
\end{remark}

\begin{lemma}\label{lem:Lip_Hol_lem_f_i}
Under Assumption~\ref{ass:other_new},
\(
\Lip f^{i;d}(t,\cdot)
+ \Holdertwo f^{i;d}
\le cd^q,
\; t\in[0,T] .
\)
\end{lemma}

We provide the following pointwise maximum deep ReLU realization lemma. This will be used in the multiple regimes maximum operator realization in the proof of Theorem~\ref{theorem:recursive_express_new}. 
\begin{lemma}[Deep ReLU realization of pointwise maximum]\label{lem:nn_realization_maximum}
For any \(M\in\mathbb N_+\) and deep ReLU networks \(\mathcal N_m^d:\mathbb R^d\to\mathbb R\), \(m=1,\dots,M\), there exists a deep ReLU network \(\mathcal N^d:\mathbb R^d\to\mathbb R\) such that
\(
\mathcal N^d(x)=\max_{1\le m\le M}\mathcal N_m^d(x),\; x\in\mathbb R^d,
\)
and
\(
\size(\mathcal N^d)\le 7(M-1)+\sum_{m=1}^M \size(\mathcal N_m^d).
\)
\end{lemma}

Then, we are now ready for value function expressivity theorem. The proof of this theorem includes the multi-level approximation and deep ReLU realization of the running payoff integral as well as regimes maximum operator. 
\begin{theorem}[Value function expressivity]\label{theorem:recursive_express_new}
Under Assumption~\ref{ass:dynamic_ass} with \( p \ge 2 \) and Assumption~\ref{ass:other_new}, for any \( \bar{p} \in [2,p] \), \( k_1,p_1 \ge 1 \), independent of \( d \), by setting sequences \( k_{n+1} = c(1+k_n) \), \( p_{n+1} = p_n+q \), \( n \in \overline{N}_1^{-1} \), there exist constants \( c_{n+1}, q_{n+1}, \tau_{n+1} \ge 1 \), \( n \in \N^{-1} \), such that for any family of probability measures \( \rho_{n+1;d} : \mathcal{B}(\mathbb{R}^d) \to \mathbb{R}_{\ge 0} \), \( n \in \N^{-1} \), satisfying \( \mathbb{M}_{\bar{p}}(\rho_{n+1;d}) \le k_{n+1} d^{p_{n+1}} \), and for any \( \varepsilon > 0 \), \( i \in \mathcal{J} \), there exist deep ReLU networks \( \hat{V}_{n+1,\varepsilon}^{i;d} : \mathbb{R}^d \to \mathbb{R} \), \( n \in \N^{-1} \), satisfying
\[
    \emph{(i)}. \; \| \hat{V}_{n+1,\varepsilon}^{i;d} - V_{n+1}^{i;d} \|_{2,\rho_{n+1;d}} \le \varepsilon, \; \; \; \emph{(ii)}. \; \size(\hat{V}_{n+1,\varepsilon}^{i;d}) +  \Growth(\hat{V}_{n+1,\varepsilon}^{i;d}) \le c_{n+1} d^{q_{n+1}} \varepsilon^{-\tau_{n+1}} .
\]
\end{theorem}

Using \Cref{theorem:recursive_express_new}, we next approximate the discrete martingale integrands. For \(i\in\mathcal J\), \(n\in\N^{-1}\), \(K\in\mathbb N_+\), let \(z_{n,k}^{i;K,d}:\mathbb R^d\to\mathbb R^d\), \(k \in \K \), satisfy
\(
z_{n,k}^{i;K,d}(X_{t_k^n}^{d})=\hat Z_{t_k^n}^{i;K,d},
\; \mathbb P\text{-a.s.},
\)
and define
\[
z_n^{i;K,d}(t,x):=\sum_{k=0}^{K-1} z_{n,k}^{i;K,d}(x)\,1_{[t_k^n,t_{k+1}^n)}(t),
\quad (t,x)\in[t_n,t_{n+1})\times\mathbb R^d.
\]
As in \cite[Section~4.2.1]{ye2025deepmartingale}, define
\[
\mu_n^{K;d}(A)
:=
\E\!\Big[\sum_{k=0}^{K-1}1_A(t_k^n,X_{t_k^n}^{d})\,\Delta t \Big],
\quad
A\in\mathcal B(\mathbb R^{1+d}) .
\]
Since the proof argument follows \cite[Theorem~4.26-4.28]{ye2025deepmartingale}, we omit it here.

\begin{theorem}[Integrand approximation \& realization]\label{theorem:joint nerual network realization_new}
Under Assumption~\ref{ass:dynamic_ass} with \(p>4\) and Assumption~\ref{ass:other_new}, for each \(n\in\N^{-1}\) there exist \(\bar c_n,\bar q_n,\bar\tau_n,\bar m_n\ge 1 \), independent of \(d\), such that for every \(K\in\mathbb N_+\), \(\varepsilon\in(0,1]\), and \(i\in\mathcal J\), there exists a deep ReLU network \(\tilde z_{n,\varepsilon}^{i;K,d}:\mathbb R^{1+d}\to\mathbb R^d\) satisfying

\vspace{0.2em}
\emph{(i)}. 
\(
\|\tilde z_{n,\varepsilon}^{i;K,d}-z_n^{i;K,d}\|_{2,\mu_n^{K;d}}\le \varepsilon ;
\)

\emph{(ii)}. for all \(t\in[t_n,t_{n+1})\),
\(
\size(\tilde z_{n,\varepsilon}^{i;K,d})
+\Growth(\tilde z_{n,\varepsilon}^{i;K,d}(t,\cdot))
\le \bar c_n d^{\bar q_n}K^{\bar m_n}\varepsilon^{-\bar\tau_n}.
\)
\end{theorem}


We can now state our main expressivity result for \textit{DeepMartingales}.

\begin{theorem}[DeepMartingale expressivity]\label{thm:express_deep_mtg}
Under the Assumption~\ref{ass:N_0_structural_ass_model_deter}, Assumption~\ref{ass:dynamic_ass} with \(p>4\),  Assumption~\ref{ass:other_new}, there exist constants \(\widetilde c,\widetilde q,\widetilde r>0\), independent of \(d\), such that for any \(\varepsilon\in(0,1]\), there exist \(K_{\varepsilon;d}\in\mathbb N_+\) with
\(
K_{\varepsilon;d}\le \widetilde c\, d^{\widetilde q}\varepsilon^{-2},
\), deep ReLU networks \(\tilde z_{n,\varepsilon}^{i;d}:\mathbb R^{1+d}\to\mathbb R^d\), \(n\in\N^{-1}\), \(i\in\mathcal J\), 
such that for \textit{DeepMartingales}
\[
\widetilde M_{t_n,\varepsilon}^{i;d}
:=
\sum_{m=0}^{n-1}\sum_{k=0}^{K_{\varepsilon;d}-1}
\tilde z_{m,\varepsilon}^{i;d}(t_k^m,X_{t_k^m}^{d})\cdot \Delta W_{t_k^m}^{d},
\quad n\in\overline N,
\]
and
\(
\widetilde M_{\varepsilon}^{i;d}:=(\widetilde M_{t_n,\varepsilon}^{i;d})_{n\in\overline N},
\; 
\widetilde M_\varepsilon^d:=(\widetilde M_{\varepsilon}^{i;d})_{i\in\mathcal J},
\)
we have, for every \(n\in\N^{-1}\),

\vspace{0.1em}
\emph{(i)}. 
\(
\Big(\E\big[\max_{i\in\mathcal J}
|\widetilde U_n^{i;d}(\widetilde M_\varepsilon^d)-\overline Y_{t_n}^{i;d}|^2\big]\Big)^{\frac{1}{2}}
\le J(N-n)\varepsilon ;
\)

\emph{(ii)}.  for all \(i\in\mathcal J\) and \(t\in[t_n,t_{n+1})\),
\(
\size(\tilde z_{n,\varepsilon}^{i;d})
+\Growth(\tilde z_{n,\varepsilon}^{i;d}(t,\cdot))
\le \widetilde c\, d^{\widetilde q}\varepsilon^{-\widetilde r}.
\)
\end{theorem}

\vspace{0.5em}
\paragraph{Expressivity Example: Affine It\^o Diffusion}
Affine It\^o diffusions provide a standard class of dynamics covered by our framework; see \cite{ye2025deepmartingale,Jentzen23} and \Cref{subsec:supplement-affine-ito-diffusion} in the Supplementary Materials for auxiliary estimates.

\begin{definition}[Affine It\^o diffusion]\label{def:AID}
If \(X^{d}\) is the unique strong solution of \eqref{eq:SDE}, 
we call \(X^{d}\) an affine It\^o diffusion if \(\mu^d\) and \(\sigma^d\) are affine, i.e., there exist
\(
A_\mu^d\in\R^{d\times d},\; b_\mu^d\in\R^d,\;
A_\sigma^{k;d}\in\R^{d\times d},\; b_\sigma^{k;d}\in\R^d,\; k=1,\dots,d,
\)
such that
\[
dX_t^d=(A_\mu^d X_t^d+b_\mu^d)\,dt+\sum_{k=1}^d (A_\sigma^{k;d}X_t^d+b_\sigma^{k;d})\,dW_t^{k;d}.
\]
\end{definition}

We impose the following Lipschitz and growth rate conditions.

\begin{as}\label{def:ADI_express}
Assume that \(X^d\) satisfies \Cref{def:AID}. Moreover, there exist \(c_a,q_a>0\) independent of $d$, such that,
\begin{itemize}
    \item[(i)] \( \|A_{\mu}^{d}\|_H \le c_{a}(\log d)^{\frac{1}{2}}\) and
    \(
    \|A_{\sigma}^{d}\|_2^2:=\sum_{k=1}^d \|A_{\sigma}^{k;d}\|_2^2 \le c_{a}(\log d)^{\frac{1}{2}};
    \)
    \item[(ii)] \( \|b_{\mu}^{d}\| \le c_{a} d^{q_a}\) and \( \|b_{\sigma}^{d}\|_{\H}\le c_{a} d^{q_a}\), where \(b_\sigma^d =(b_\sigma^{1;d},\dots,b_\sigma^{d;d})\).
\end{itemize}
\end{as}

Assumption~\ref{def:ADI_express} covers, e.g., Geometric Brownian Motion, Ornstein--Uhlenbeck dynamics; see \cite[Remark~4.34]{ye2025deepmartingale}. The \textit{DeepMartingales} expressivity result for affine It\^o diffusions then follows by an argument analogous to \cite[Proof of Theorem~4.36]{ye2025deepmartingale}, combined with \Cref{lemma:AID-log_ass1} in Supplementary Materials and Remark~\ref{rem:reward_ass_express_to_ass_integration}. We therefore omit the proof.

\begin{corollary}[DeepMartingale expressivity for affine It\^o diffusion]\label{thm:AID-log_express}
If \(X^d\) satisfies Assumption~\ref{def:ADI_express}, and \(f^{i;d},\Phi^{i;d},l_{ij}^d\), \(i,j\in\mathcal J\), satisfy Assumption~\ref{ass:other_new}, then Theorem~\ref{thm:express_deep_mtg} holds.
\end{corollary}


\subsection{Connection with ``Delta''}\label{subsec:delta_deephedge}

Dual martingale methods are closely related to delta hedging and delta risk; see, e.g., \cite{belome09,roger10,puredual-mf}. In our setting, this relation is immediate from the continuation and Doob martingales representation from \eqref{eq:BSDE_system}. 

\begin{proposition}[Delta representation]\label{lem:delta_deephedge}
Fix \(i\in\mathcal J\) and \(n\in\N^{-1}\), and define
\(
\widetilde u_n^{i;d}(t,x):=\E\big[V_{n+1}^{i;d}(X_{t_{n+1}}^{t,x;d})\big],
\; (t,x)\in[t_n,t_{n+1}]\times\R^d.
\)
If \(\widetilde u_n^{i;d}\in C^{1,2}([t_n,t_{n+1}]\times\R^d)\), then
\(
\overline Z_t^{i;d}
=
(\nabla_x \widetilde u_n^{i;d}\,\sigma^d)(t,X_t^{t_n,x;d}),
\; t\in[t_n,t_{n+1}],
\)
where \(\overline Z^{i;d}\) is the Doob martingale integrand in \eqref{eq:BSDE_system}. If, in addition, \(\sigma^d(t,x)\) is invertible, then the delta hedge ratio is
\[
\Pi_t^{i,n;d}
:=
(\sigma^d)^{-1}(t,X_t^{t_n,x;d})\,\overline Z_t^{i;d}
=
\nabla_x \widetilde u_n^{i;d}(t,X_t^{t_n,x;d}).
\]
\end{proposition}

Hence the \textit{DeepMartingale} integrand \(z_n^{i,\theta_n^i;d}\) may be viewed as a deep delta hedge, namely
\(
(\sigma^d)^{-1}(t,X_t^{t_n,x;d})\,z_n^{i,\theta_n^i;d}(t,X_t^{t_n,x;d}),
\)
whenever \(\sigma^d(t,x)\) is invertible.


\section{Numerical Experiments}\label{sec:numerical}



We first describe the implementation of the \textit{DeepMartingale} dual solver and then present two benchmark studies. Throughout this section, we use \textit{DeepPD} to denote the overall numerical framework consisting of this dual solver together with an auxiliary deep policy-based approach for computing feasible lower bounds and empirical upper--lower gaps.

\subsection{DeepPD: dual implementation and lower bound benchmark}




On the dual side, we train the \textit{DeepMartingale} family $M^{\theta;K}$ from \eqref{deep_mart}. A key computational feature is that we optimize the dual problem only for one chosen reference regime $i_0 \in \mathcal J$, which empirically already yields accurate upper bounds for all regimes. Heuristically, this is consistent with our duality theory, since the Doob martingales are simultaneously optimal martingale penalties across regimes. This reference-regime training substantially reduces the computational burden while preserving the quality of the resulting upper bounds in our experiments.

The expressivity results in Section~\ref{subsec:converge_express_ReLU} also motivate a dimension scaling mechanism, which we do not elaborate on here for brevity; see \cite[Section~5.1.3]{ye2025deepmartingale} for the analogous stopping case.

For numerical benchmark, we also compute feasible lower bounds by implementing a deep policy-based approach adapted from the idea of \cite[\textit{Deep Impulse Control}]{Jia-Wong01022024} within our primal dynamic programming principle. Concretely, we parameterize the regime decision in the primal recursion \eqref{eq:primal_DPP_indicator} by a softmax network, extract the induced hard switching rule, and evaluate the resulting admissible strategy out of sample. Since this lower bound construction plays only a benchmarking role in the present paper, we omit further implementation details.



Given the SDE coefficients \((\mu,\sigma)\) in \eqref{eq:SDE} and the payoff data \((f^i,l_{ij},\Phi^i)\), the resulting deep dual algorithm is summarized in Algorithm~\ref{Alg:1}, which is the dual training component of \textit{DeepPD}.


\begin{algorithm}[H]
\caption{\textit{DeepMartingale} dual solver for optimal switching}\label{Alg:1}
\small
\begin{algorithmic}[1]
\REQUIRE Intervention time grid $\pi$, subgrid size $K$, training batch size $B^{\mathrm{tr}}$, number of epochs $M$, reference regime $i_0 \in \mathcal J$, baseline $(\eta_n^{i_0})_{n\in\overline{N}^{-1}}$.
\STATE Initialize the DeepMartingale parameters $\theta = (\theta_n)_{n\in\overline{N}^{-1}}$ in \eqref{deep_mart}.
\FOR{$m = 1,\dots,M$}
    \STATE Simulate $B^{\mathrm{tr}}$ sample paths of $(X,W)$ with step size $\Delta t = T/(NK)$.
    \STATE Set $U_N^i = \Phi^i(X_{t_N})$, $i\in\mathcal J$.
    \FOR{$n = N-1,\dots,0$}
        \STATE Update $\theta_n$ by solving \eqref{eq_optim_second_moment} or \eqref{eq_optim} in Problem~\ref{pb:backward_minimization} at the reference regime $i_0$.
        \STATE Update $U_n^i$, $i\in\mathcal J$, via the dual recursion \eqref{eq:equality-dual-upper-bound}.
    \ENDFOR
\ENDFOR
\STATE \textbf{return} $\theta$.
\end{algorithmic}
\end{algorithm}

We use ReLU activations and apply batch normalization before the input layer and activations. Unless stated otherwise, we take depth $I=3$, training batch size $B^{\mathrm{tr}}=4{,}096$, Adam optimizer with learning rate $10^{-3}$, and Xavier normal initialization. The number of training epochs is $M=1000+20d$ for the first example and $M=300+3d$ for the second one. We always choose $i_0=1$ for duality training. Final upper and lower bounds are evaluated with $1{,}638{,}400$ new samples.\footnote{Codes are available at: \url{https://github.com/GEOR-TS/DeepMartingale-OptimalSwitching}.}

For a like-for-like comparison, we re-implement \cite[\textit{DeepOSJ}]{Bayraktar23-deep-switching} in our upper-lower bound evaluation by Pytorch for the first example and keep all original training setup, except for model changes and continuous-observation adjustment. For the second one, we use original codes in \cite{Bayraktar23-deep-switching}\footnote{Their codes are available at: \url{https://github.com/april-nellis/osj}.} and implement our upper-lower bound evaluation. 

\subsection{Experiments}\label{subsubsec:discrete_obs}

\vspace{0.5em}
\paragraph{Continuous-observation under geometric Brownian motion}
We fix \(\mathcal J=\{1,2,3\}\), \(T=1\), \(N=12\), terminal payoff \(\Phi^i\equiv0\), and running rewards
\(
f^1(t,x)=-0.5 \)
\(
f^2(t,x)=\frac{2}{d}\sum_{k=1}^d x_k-100,
\)
\(
f^3(t,x)=2(x_1-1.1x_d)-1,
\)
with switching costs
\(
l_{ij}(t,x)\equiv 0.2\,|i-j|.
\)
The state process is the \(d\)-dimensional geometric Brownian motion
\[
\frac{dX_t}{X_t}
=
-0.05\,\mathbf 1_d\,dt+\Sigma\,dW_t,
\qquad
X_0=50\,\mathbf 1_d,
\]
where \(\Sigma=\diag(\sigma_1,\ldots,\sigma_d)\), with \(\sigma_k=0.2\) for \(k\le d/2\) and \(\sigma_k=0.3\) otherwise.

To handle the continuous-observation integral, we use \(K=60+d\) substeps between intervention dates. Since \(i_0=1\), \(f^1\equiv -0.5\), and \(l_{1j}\le 0.4\), we choose the baseline
\(
\eta_n^{i_0}=0.45(n-N)
\)
and use the $L^2$-surrogate loss \eqref{eq_optim_second_moment}. Table~\ref{tab:GBM_contin_obs_combined} compares \textit{DeepPD} with a continuous-observation version of \textit{DeepOSJ}. Both methods are implemented in PyTorch in single precision (float32) on an NVIDIA A100 GPU (40 GB memory) with dual AMD Rome 7742 CPUs.

Table~\ref{tab:GBM_contin_obs_combined} reports upper bounds, lower bounds, maximal duality gaps across regimes, and the CVaR of the hedging portfolio for regime \(i=1\). \textit{DeepOSJ} is slightly better at \(d=2\), but from \(d=10\) onward \textit{DeepPD} yields smaller gaps and substantially better tail-risk performance. In particular, the maximal duality gap of \textit{DeepPD} stays close to \(0.1\) across all tested dimensions, while \textit{DeepOSJ} runs out of memory for \(d\ge 20\). Thus, \textit{DeepPD} remains stable and accurate up to \(d=100\). 
This highlights the main computational strength of 
the proposed dual solver: 
the reference-regime \textit{DeepMartingale} training is memory-efficient, dimension-scalable, and produces accurate computable upper bounds  
together with robust hedging performance, while the auxiliary lower bounds provide an empirical benchmark for gap assessment. 
Figure~\ref{fig:GBM_hedging_contin} shows the worst-case hedging error distribution for \(d=10\); \textit{DeepPD} exhibits smaller VaR and lighter tails.



Since the dual upper-bound operator in \eqref{eq:upper_bound_operator} also induces a non-adapted switching rule through its maximizing index, we generate $200,000$ out-of-sample states to visualize the resulting preferred-regime partitions for \(d=2\) and \(n=6\); see Figure~\ref{fig:preferred_regime_region}. We compare \textit{DeepPD}, using both the primal policy and the dual-induced rule, with \textit{DeepOSJ}, where switching decisions are determined by the rule in \cite{Bayraktar23-deep-switching}.

Two observations are worth emphasizing. First, for all current regimes, the partitions induced by the \textit{DeepPD} dual are qualitatively close to those obtained from the \textit{DeepPD} primal. Since the dual-induced rule uses future information through the martingale noise, it is not necessarily admissible, and its boundary is therefore slightly more diffuse. Nevertheless, the overall geometry remains highly consistent, which supports the interpretation that the learned dual martingale captures the correct switching structure. If the learned \textit{DeepMartingales} coincide with the Doob martingales, then Theorem~\ref{thm:stong_duality} implies that the dual-induced boundary recovers the exact switching boundary.

Second, compared with \textit{DeepOSJ}, the \textit{DeepPD} dual produces a more coherent and stable partition, whereas \textit{DeepOSJ} exhibits more pronounced kinks and local distortions near the switching region. This suggests that the dual representation captures the switching geometry more robustly.




\begin{table}[ht]
\caption{Continuous observation under GBM}
\label{tab:GBM_contin_obs_combined}
\centering
\small
\setlength{\tabcolsep}{4pt}
\renewcommand{\arraystretch}{1.1}
\begin{tabular}{c c c c c c c}
\hline
\multirow{2}{*}{\raisebox{-0.9ex}{$d$}} 
& \multirow{2}{*}{\raisebox{-0.9ex}{Method}} 
& \multirow{2}{*}{\raisebox{-0.9ex}{UB}} 
& \multirow{2}{*}{\raisebox{-0.9ex}{LB}} 
& \multirow{2}{*}{\raisebox{-0.9ex}{Gap(max)}} 
& \multicolumn{2}{c}{CVaR (\(i=1\))} \\
\cmidrule(lr){6-7}
& & & & & \(95\%\) & \(99\%\) \\
\hline
\multirow{2}{*}{$ 2 $}
& DeepPD  & $ [7.191, 7.261, 7.069] $ & $ [7.084, 7.150, 6.950] $ & $ 0.115 $ & $2.731$ & $\mathbf{3.855}$ \\
& DeepOSJ & $ [7.158, 7.206, 7.006] $ & $ [7.038, 7.106, 6.906] $ & $ 0.120 $ & $\mathbf{2.014}$ & $4.324$ \\
\hline
\multirow{2}{*}{$ 10 $}
& DeepPD  & $ [5.098, 5.155, 4.959] $ & $ [5.009, 5.063, 4.863] $ & $ 0.096 $ & $\mathbf{2.510}$ & $ \mathbf{3.478} $ \\
& DeepOSJ & $ [5.098, 5.143, 4.943] $ & $ [4.935, 4.968, 4.768] $ & $ 0.175 $ & $4.310$ & $7.023$ \\
\hline
\multirow{2}{*}{$ 20 $}
& DeepPD  & $ [4.701,4.752,4.555] $ & $ [4.609,4.653,4.453] $ & $ 0.103 $ & $\mathbf{2.566}$ & $\mathbf{3.625}$ \\
& DeepOSJ & N/A & N/A & N/A & N/A & N/A \\
\hline
\multirow{2}{*}{$ 30 $}
& DeepPD  & $ [4.552, 4.598, 4.401] $ & $ [4.456, 4.491, 4.291] $ & $ 0.109 $ & $\mathbf{2.526}$ & $\mathbf{3.571}$ \\
& DeepOSJ & N/A & N/A & N/A & N/A & N/A \\
\hline
\multirow{2}{*}{$ 50 $}
& DeepPD  &  $ [4.433, 4.469, 4.271] $ & $ [4.336, 4.357, 4.157] $  & $ 0.114 $ & $\mathbf{2.529}$ & $\mathbf{3.598}$ \\
& DeepOSJ & N/A & N/A & N/A & N/A & N/A \\
\hline
\multirow{2}{*}{$ 100 $}
& DeepPD  & $ [4.348, 4.366, 4.169] $ & $ [4.253, 4.250, 4.050] $ & $ 0.119  $ & $\mathbf{2.707}$ & $\mathbf{3.912}$ \\
& DeepOSJ & N/A & N/A & N/A & N/A & N/A \\
\hline
\end{tabular}
\end{table}

\begin{figure}[tbp]
\begin{minipage}[t]{0.48 \linewidth}
\centering
\includegraphics[height=5.5cm,width=5.4cm]{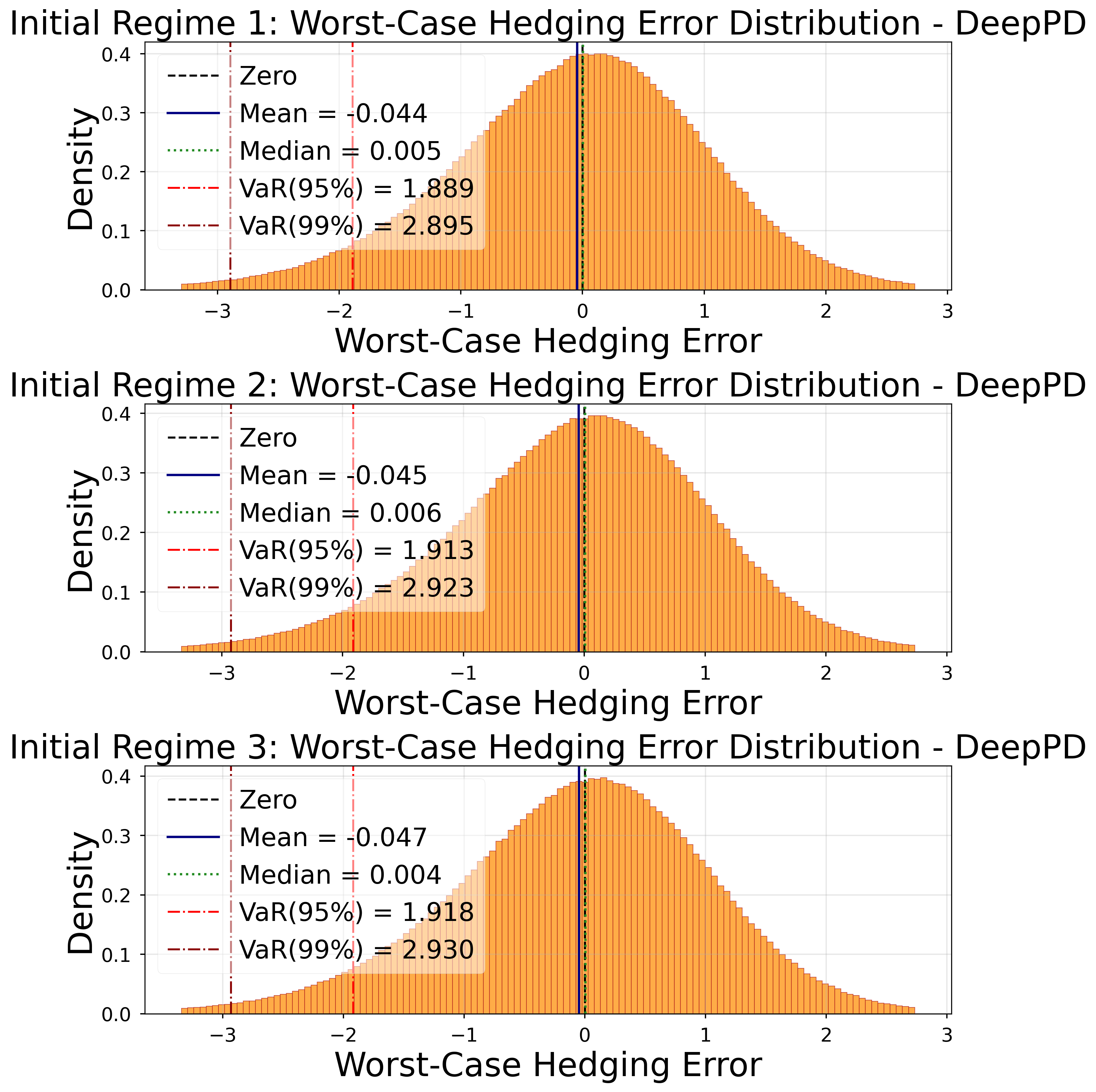}
\end{minipage}%
\begin{minipage}[t]{0.48 \linewidth}
\centering
\includegraphics[height=5.5cm,width=5.3cm]{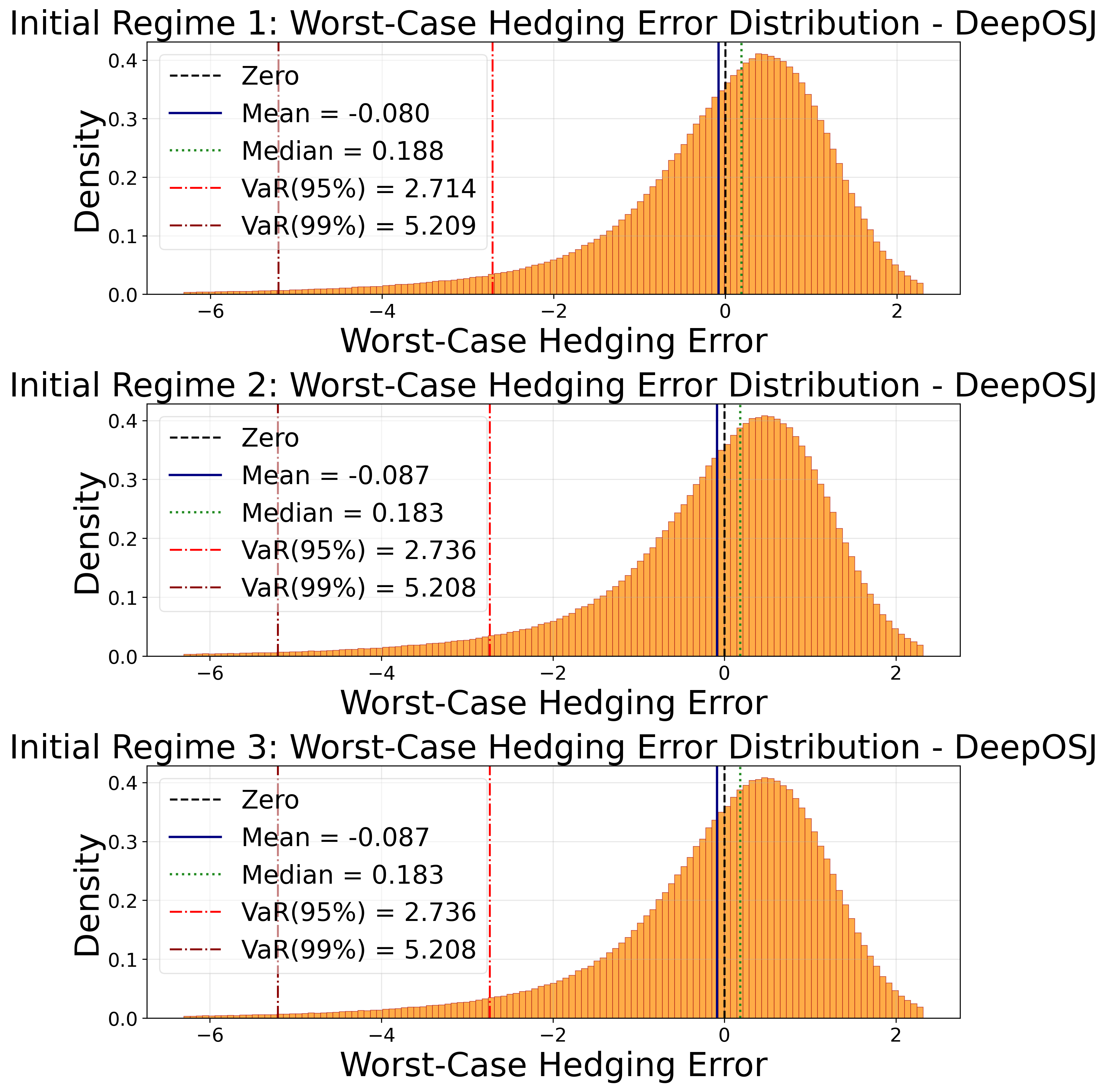}
\end{minipage}%
\caption{Worst-case hedging error distribution, \(d=10\)}
\label{fig:GBM_hedging_contin}
\end{figure}

\begin{figure}[tbp]
\centering
\includegraphics[height=9cm,width=9.2cm]{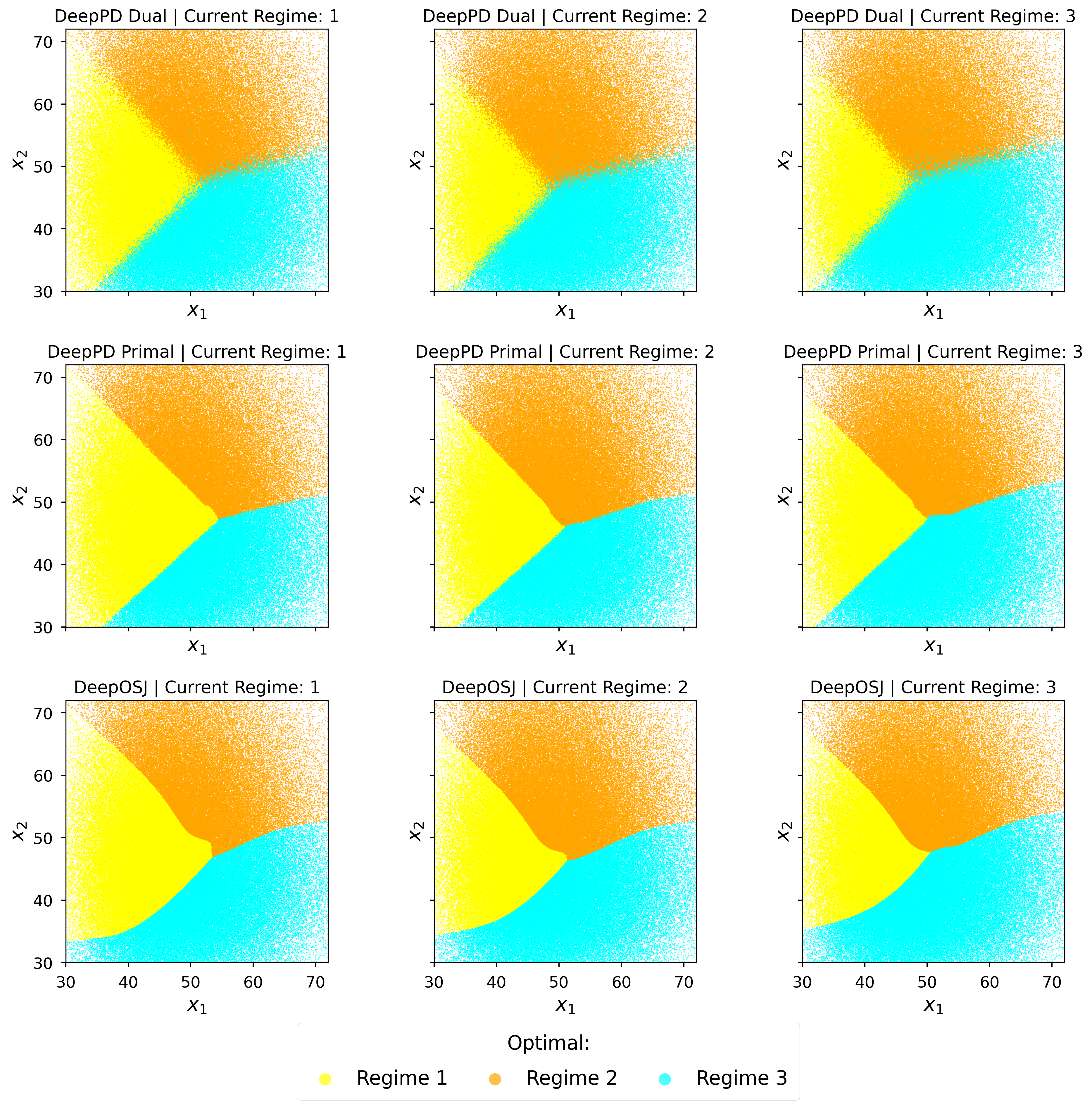}
\caption{Switching Region/Boundary Comparison, $d=2$, $n=6$}
\label{fig:preferred_regime_region}
\end{figure}

\vspace{0.5em}
\paragraph{Brownian--Poisson filtration}
We test our duality theory in Brownian--Poisson filtration and extend \textit{DeepMartingales} \eqref{deep_mart} by an additional jump-network term
\[
\sum_{k=0}^{K-1} z_{n,\mathrm{P}}^{i,\theta_n;K,d}(t_k^n,X_{t_k^n})\,\Delta N_{t_k^n},
\quad n\in\N^{-1},
\]
where \(N\) is a $d$-dimensional Poisson process independent of \(W\). 

Following \cite[Example~4.1]{Bayraktar23-deep-switching}, we consider the exponential OU model with jumps
\begin{equation}\label{eq:exp_OU_jumps}
d(\log X_t)=\kappa(\mu-\log X_t)\,dt+\Sigma^1\,dW_t+\Sigma^2\,dN_t,
\qquad X_0=x,
\end{equation}
where 
\(\Sigma^1\in\R^{d\times d}\) is non-degenerate, \(\Sigma^2\) is a \(\R^{d\times d}\)-valued random variable, \(\kappa\in\R^{d\times d}\) and \(\mu\in\R^d\). In this example, we fix $ K \equiv 1 $ to match the setup in \cite{Carmona-Ludkoviski01122008,Bayraktar23-deep-switching}.

We use the parameter specification of \cite[Example~4.1]{Bayraktar23-deep-switching} and compare \textit{DeepPD}, \textit{DeepOSJ}, and the least-squares benchmark \cite[\textit{LS}]{Carmona-Ludkoviski01122008}. We test both the upper-bound loss \eqref{eq_optim} and $L^2$-surrogate loss \eqref{eq_optim_second_moment}. Since \(i_0=1\), \(f^1\equiv -1\), and
\[
l_{1j}(s)\le \frac{0.01}{d-1}\sum_{k=2}^d X_s^k + 0.001,
\]
we choose
\(
\eta_n^{i_0}
=
(n-N)\big(0.01 E_{t_n}+0.001+\frac{1}{720}\big),
\; 
E_{t_n}
:=
e^{0.02}\max\big\{\frac{1}{d-1}\sum_{k=2}^d X_{t_n}^k,6\big\},
\)
using the explicit conditional moment bound from \cite{Carmona-Ludkoviski01122008,Bayraktar23-deep-switching}. All methods in this example are run in PyTorch on an Apple Silicon M4 Pro CPU with 64 GB memory.



Figure~\ref{fig:exp_OU_value_compare} shows that \textit{DeepPD} remains competitive in the Brownian--Poisson setting. \textit{DeepOSJ} attains the best upper bound, whereas \textit{DeepPD} typically gives the stronger feasible lower bound. Even when the primal approximation is imperfect (e.g.\ $d=50$), the \textit{DeepOSJ} upper bound remains accurate, demonstrating the robustness of duality method in high dimensions. We use the lower bound mainly as a numerical benchmark for the dual solver. The advantage of \textit{DeepPD} on the dual side becomes more pronounced when observation is more frequent as shown in the Brownian setting.

\begin{figure}[tbp]
\centering
\includegraphics[height=4cm,width=11cm]{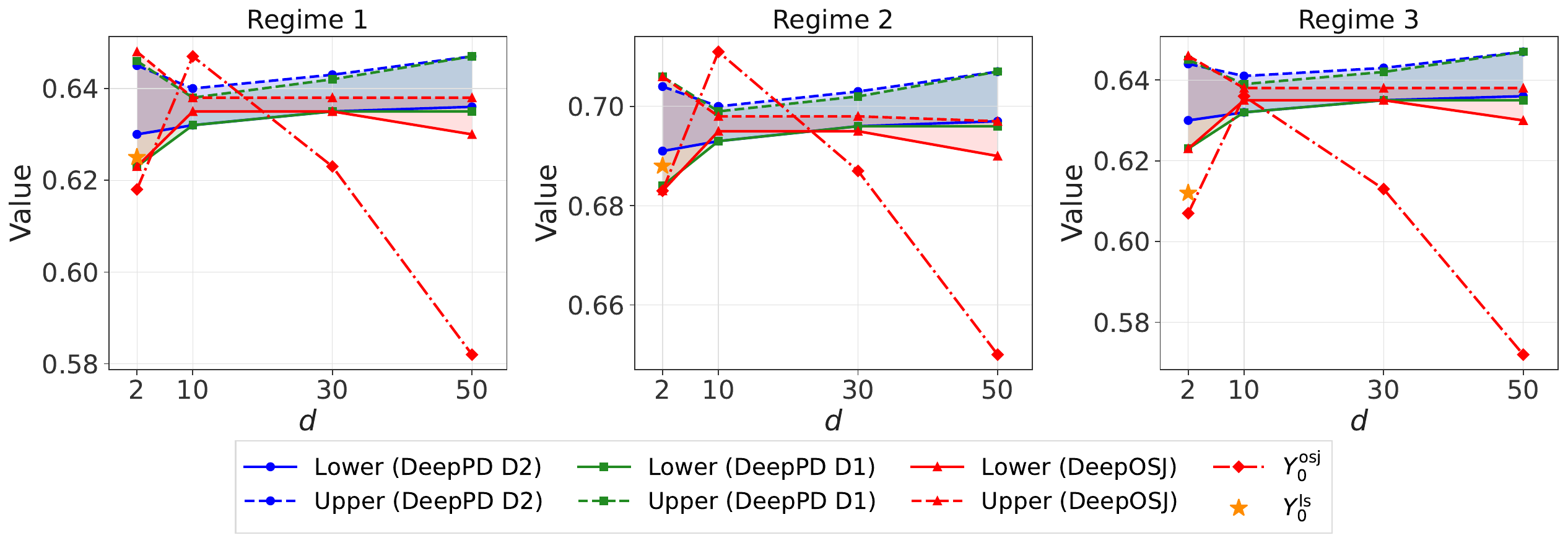}
\caption{Brownian--Poisson filtration: value comparison}
\label{fig:exp_OU_value_compare}
\end{figure}



\appendix
\section{Proofs} 

\subsection{Proof of results in Section~\ref{sec:pb_formulate}}

\begin{proof}[Proof of Theorem~\ref{thm:switching-equiv-mode-decide-primal}]
The case \(n=N\) is immediate. Assume \(n<N\).

\textit{Step 1: switching \(\Rightarrow\) regime-decision.}
Fix \(\alpha=(\tau_r,d_r)_{r\ge 0}\in\mathcal A_n^i\). Define the induced regime-decision process \(j=(j_m)_{m=n}^N\) by
\[
j_m:=\sum_{r\ge 0} d_r\,1_{\{\tau_r\le m<\tau_{r+1}\}},
\quad m\in\N_n^{-1},
\quad
j_N:=j_{N-1},
\]
with the convention that \(\tau_{r+1}=N\) once \(\tau_r=N\). Since each \(\tau_r\) is a discrete stopping time, \(j_m\) is \(\mathcal F_{t_m}\)-measurable for every \(m\), and hence \(j\in\mathcal D_n^i\). By regrouping the running rewards and switching costs, we obtain
\(
J_n^i(\alpha)=L_n^i(j).
\)
Taking \(\mathbb E_{t_n}[\cdot]\) and then the essential supremum over \(\alpha\in\mathcal A_n^i\) yields
\(
\overline Y_{t_n}^i
\le
\operatorname*{ess\,sup}_{j\in\mathcal D_n^i}L_n^i(j).
\)

\textit{Step 2: regime-decision \(\Rightarrow\) switching.}
Conversely, fix \(j=(j_m)_{m=n}^N\in\mathcal D_n^i\). Define recursively
\(
\tau_0:=n,\; d_0:=i,
\)
and, for \(r\ge 1\),
\[
\tau_r
:=
\inf\{k\in\overline N_n:\ k>\tau_{r-1},\ j_k\neq j_{\tau_{r-1}}\}\wedge N,
\quad
d_r:=j_{\tau_r}.
\]
Once \(\tau_{r-1}=N\), set \(\tau_r:=N\) and \(d_r:=j_{N-1}\). By adaptedness of \(j\), each \(\tau_r\) is a discrete stopping time, so \(\alpha=(\tau_r,d_r)_{r\ge 0}\in\mathcal A_n^i\). Again, regrouping gives
\(
J_n^i(\alpha)=L_n^i(j).
\)
Taking the essential supremum over \(j\in\mathcal D_n^i\) gives the reverse inequality.
\end{proof}

\vspace{0.1em}
\begin{proof}[Proof of Lemma~\ref{lem:weak-duality}]
For any \(d\in\mathcal D_n^i\), viewing \(d\) as an element in \(\mathcal J_n\),
\(
L_n^i(d)
=
\mathbb E_{t_n}\!\big[\widetilde U_n^{i,d}(M)\big]
\le
\mathbb E_{t_n}\!\big[\widetilde U_n^i(M)\big],
\)
since \(d\) is adapted and the martingale increments have zero conditional expectation, it is easy to derive \eqref{eq:mode-choosing-weak-dual} using \eqref{eq:original-primal-mode-choose}.
\end{proof}

\subsection{Proof of results in Section~\ref{sec:duality_equiv}}

\begin{proof}[Proof of Lemma~\ref{lem:no-free-lunch-dual}]
For \(n\in\N^{-1}\), we have 
\[
\overline U_{n,m}^i( \overline{M}^i) =
\int_{t_n}^{t_m} f^i(s)\,ds
+\mathcal R_m^{i,\iota_m^i}1_{(m<N)}
+\Phi^i1_{(m=N)}
-\overline M_{t_m}^i+\overline M_{t_n}^i,
\quad m\in\overline N_n
\]
by \eqref{eq:iterative_dual_upper_two_regime}, \eqref{eq:mode-decision-distinguish-def}. Then, according to \eqref{eq:dual-stopping}, 
\(
\overline Y_{t_n}^i=\max_{m\in\overline N_n}\overline U_{n,m}^i( \overline{M}^i) 
\)
and
\(
\tau_n^i=\inf\;\argmax_{m\in\overline N_n}\overline U_{n,m}^i( \overline{M}^i) .
\)
Hence, 
\begin{equation}\label{eq:stopped-Y-ori-short}
\overline Y_{t_n}^i
=
\int_{t_n}^{t_{\tau_n^i}} f^i(s)\,ds
+\mathcal R_{\tau_n^i}^{i,\iota_{\tau_n^i}^i}1_{S^{i,n}}
+\Phi^i1_{(\tau_n^i=N)}
-\overline M_{t_{\tau_n^i}}^i+\overline M_{t_n}^i.
\end{equation}

Fix \(m<N\), and set
\(
\hat\tau_m^i:=\tau_m^{\iota_m^i},
\)
\(
\hat\iota_m^i:=\iota_m^{\iota_m^i}.
\)
On \(A^{i,m}=\{\hat\tau_m^i=m,\ \hat\iota_m^i\neq i\}\), \eqref{eq:stopped-Y-ori-short} at time \(m\) in regime \(\iota_m^i\) gives
\[
\overline{\mathcal R}_n^i = \mathcal R_m^{i,\iota_m^i}
=\overline Y_{t_m}^{\hat\iota_m^i}
-l_{\iota_m^i\hat\iota_m^i}(t_m)-l_{i\iota_m^i}(t_m)
<\overline Y_{t_m}^{\hat\iota_m^i}-l_{i\hat\iota_m^i}(t_m)
\le \mathcal R_m^{i,\iota_m^i} = \overline{\mathcal R}_n^i,
\]
which can not hold with positive probability. Therefore \(\P(A^{i,m})=0\) for all \(m\in\N^{-1}\). Since
\(1_{(m<N)}=1_{\widetilde A_m^{i,n}}\) a.s.,
\eqref{eq:iterative-dual-surely-optimal-2} and
\eqref{eq:iterative-dual-surely-optimal-stopping-time-2} follow.

Next, we have
\[
B^{i,n}\cap S^{i,n}
=\bigcup_{m=n}^{N-1}\Big(A^{i,m}\cap\{\tau_n^i=m\}\Big)
\cup\big(C^{i,n}\cap S^{i,n}\big)
\subset \bigcup_{m=n}^{N-1}A^{i,m}\cup\big(C^{i,n}\cap S^{i,n}\big).
\]
Thus it remains to prove \(\P(C^{i,n}\cap S^{i,n})=0\). On this event, with
\(
\tau:=\tau_n^i<N,
\)
\(
\iota:=\iota_\tau^i\neq i,
\)
we have \(\tau_\tau^\iota=\tau\) and \(\iota_\tau^\iota=i\). Hence, by
\eqref{eq:stopped-Y-ori-short} and Supplementary Materials-\eqref{eq:stopping-equality-dual},
\begin{align*}
\overline Y_{t_n}^i
&=\int_{t_n}^{t_\tau} f^i(s)\,ds
+\mathcal R_\tau^{i,\iota}
-\overline M_{t_\tau}^i+\overline M_{t_n}^i\\
&=\int_{t_n}^{t_\tau} f^i(s)\,ds
+\overline Y_{t_\tau}^i
-l_{\iota i}(t_\tau)-l_{i\iota}(t_\tau)
-\overline M_{t_\tau}^i+\overline M_{t_n}^i\\
&<\int_{t_n}^{t_\tau} f^i(s)\,ds
+\overline Y_{t_\tau}^i-l_{ii}(t_\tau)
-\overline M_{t_\tau}^i+\overline M_{t_n}^i
=\overline Y_{t_n}^i,
\end{align*}
which can not hold with positive probability again. Thus \(\P(C^{i,n}\cap S^{i,n})=0\), and therefore
\(
\P(B^{i,n}\cap S^{i,n})=\P(D^{i,n})=0.
\)
Finally, \(1_{S^{i,n}}=1_{(B^{i,n})^c\cap S^{i,n}},\; \P\)-a.s., and
\eqref{eq:stopped-Y} follows from \eqref{eq:stopped-Y-ori-short}.
\end{proof}

\vspace{0.1em}
\begin{proof}[Proof of Theorem~\ref{thm:construction-optimal-mode-decision}]
Set \(\hat\tau_n^i:=\tau_n^{\iota_n^i}\).

\emph{Step 1: proof of (ii)--(iii), assuming the family is well defined.}
On \(\{\tau_n^i=n\}\), \(\mathbb P(D^{i,n})=0\) by \Cref{lem:no-free-lunch-dual}, hence \(\hat\tau_n^i\ge n+1\) a.s.; the second branch of \eqref{eq:construction-mode-decision} then gives
\(
j_k^{i,n}=j_k^{\iota_n^i,n},\; k\in\overline N_n,
\)
which is (iii). On \(\{\tau_n^i>n\}\), \Cref{lem:DPP-dual-stopping-optimal} in Supplementary Materials yields \(\tau_n^i=\tau_{n+1}^i\). If \(\tau_{n+1}^i>n+1\), the first branch of \eqref{eq:construction-mode-decision} gives \(j_k^{i,n}=j_k^{i,n+1}\) for \(k\ge n+1\); if \(\tau_{n+1}^i=n+1\), then (iii) at time \(n+1\) yields \(j^{i,n+1}=j^{\iota_{n+1}^i,n+1}\), so again \(j_k^{i,n}=j_k^{i,n+1}\). Thus
\(
j_k^{i,n}=j_k^{j_n^{i,n},\,n+1},\; k\in\overline N_{n+1},
\)
on \(\{\tau_n^i>n\}\). On \(\{\tau_n^i=n\}\), combine (iii) with the previous argument on each \(\{\iota_n^i=\ell\}\), \(\ell\in\mathcal J\), to obtain the same identity. Hence (ii).

\emph{Step 2: well-definedness, adaptedness, membership in \(\mathcal D_n^i\), and (i).}
We argue by backward induction on \(n\). The case \(n=N\) is immediate since \(j_N^{i,N}=i\).

Fix \(n<N\) and assume that, for every \(m\in\overline N_{n+1}\) and \(i\in\mathcal J\), \(j^{i,m}\) is well defined, \(\mathbb F_m\)-adapted, belongs to \(\mathcal D_m^i\), and satisfies (i). Let \(i\in\mathcal J\).

On \(\{\tau_n^i>n\}\), \eqref{eq:construction-mode-decision} only uses already constructed \(j^{\ell,m}\) with \(m\ge n+1\), so \(j^{i,n}\) is well defined. Moreover, for \(k\in\overline N_n\),
\[
j_k^{i,n}
= i\,1_{\{k<\tau_n^i\}}
+ \sum_{m=n+1}^k\sum_{\ell=1}^J
j_k^{\ell,m}\,1_{\{\tau_n^i=m,\ \iota_m^i=\ell\}},
\]
hence \(j_k^{i,n}\in\mathcal F_{t_k}\). Since \(j_n^{i,n}=i\), property (ii) gives \(j_k^{i,n}=j_k^{i,n+1}\) for \(k\ge n+1\); thus \(j_N^{i,n}=j_{N-1}^{i,n}\) by the induction hypothesis, so \(j^{i,n}\in\mathcal D_n^i\). Finally, \Cref{lem:DPP-dual-stopping-optimal} in Supplementary Materials gives
\[
\overline Y_{t_n}^i>\overline{\mathcal R}_n^i,
\quad
j_n^{i,n}=i\in\argmax_{j\in\mathcal J}\big(\mathcal R_n^{i,j}1_{(n<N)}\big),
\]
and the remaining statement in (i) are from the induction hypothesis via (ii).

On \(\{\tau_n^i=n\}\), \(\mathbb P(D^{i,n})=0\) implies \(\hat\tau_n^i\ge n+1\) a.s., so \eqref{eq:construction-mode-decision} again only refers to already constructed \(j^{\ell,m}\), \(m\ge n+1\). Also,
\[
j_k^{i,n}
=\iota_n^i\,1_{\{k<\hat\tau_n^i\}}
+\sum_{m=n+1}^k\sum_{\ell=1}^J
j_k^{\ell,m}\,1_{\{\hat\tau_n^i=m,\ \iota_m^{\iota_n^i}=\ell\}},
\]
which is \(\mathcal F_{t_k}\)-measurable after partitioning over \(\{\iota_n^i=p\}\), \(p\in\mathcal J\). By (iii), \(j^{i,n}=j^{\iota_n^i,n}\); on each \(\{\iota_n^i=p\}\), the previous case applies to the deterministic regime \(p\), since \(\tau_n^p=\hat\tau_n^i\ge n+1\), and therefore \(j_N^{i,n}=j_{N-1}^{i,n}\). Hence \(j^{i,n}\in\mathcal D_n^i\). Moreover,
\[
j_n^{i,n}=\iota_n^i\in\argmax_{j\in\mathcal J}\big(\mathcal R_n^{i,j}1_{(n<N)}\big),
\quad
\overline Y_{t_n}^{i}=\overline{\mathcal R}_n^{i},
\quad
\overline Y_{t_n}^{\iota_n^i}>\overline{\mathcal R}_n^{\iota_n^i},
\]
by \Cref{lem:DPP-dual-stopping-optimal} in Supplementary Materials, and the remaining claims in (i) follow from \(j^{i,n}=j^{\iota_n^i,n}\) and the previous case.
The induction completes.
\end{proof}

\vspace{0.1em}
\begin{proof}[Proof of Theorem~\ref{thm:expansion-surely-optimal}]
We argue by backward induction on \(n\). For \(n=N\), since \(j_N^{i,N}=i\) and \(l_{ii}\equiv0\),
\(
\widetilde U_N^{i,j^{i,N}}(\overline M)
=\Phi^{j_N^{i,N}}-l_{i\,j_N^{i,N}}(t_N)
=\Phi^i
=\overline Y_{t_N}^i.
\)
Fix \(n<N\) and assume that, for all \(m\in\overline N_{n+1}\) and \(i\in\mathcal J\),
\(
\overline Y_{t_m}^i=\widetilde U_m^{i,j^{i,m}}(\overline M)
\;  \mathbb P\text{-a.s.}
\)
Let \(i\in\mathcal J\).

\emph{Case 1: \(\tau_n^i>n\).}
Then \(j_k^{i,n}=i\) for \(n\le k<\tau_n^i\) and
\(j_k^{i,n}=j_k^{\iota_{\tau_n^i}^i,\tau_n^i}\) for \(k\ge\tau_n^i\).
Moreover, by \Cref{lem:no-free-lunch-dual},
\(
\tau_{\tau_n^i}^{\iota_{\tau_n^i}^i}>\tau_n^i
\; \text{on }\{\tau_n^i<N\},\ \mathbb P\text{-a.s.}
\)
Hence, by \Cref{thm:construction-optimal-mode-decision}(i),
\(
j_{\tau_n^i}^{\iota_{\tau_n^i}^i,\tau_n^i}
\in\argmax_{j\in\mathcal J}\mathcal R_{\tau_n^i}^{\iota_{\tau_n^i}^i,j}
=\{\iota_{\tau_n^i}^i\}
\;\text{on }\{\tau_n^i<N\},
\)
so \(j_{\tau_n^i}^{i,n}=\iota_{\tau_n^i}^i\) there. Therefore, expanding 
\(\widetilde U_n^{i,j^{i,n}}(\overline M)\) at \(\tau_n^i\),
\[
    \begin{aligned}
        & \widetilde U_n^{i,j^{i,n}}(\overline M) = \\
& 
\int_{t_n}^{t_{\tau_n^i}} f^i(s) ds
-\overline M_{t_{\tau_n^i}}^i+\overline M_{t_n}^i
+ \big[
\widetilde U_{\tau_n^i}^{\iota_{\tau_n^i}^i, j^{\iota_{\tau_n^i}^i,\tau_n^i}}(\overline M)
-l_{i \iota_{\tau_n^i}^i}(t_{\tau_n^i})
 \big] 1_{(\tau_n^i<N)}
+\Phi^i1_{(\tau_n^i=N)}.
    \end{aligned}
\]
Since \(\tau_n^i\in\overline N_{n+1}\) on \(\{\tau_n^i>n\}\), the induction hypothesis,
applied on the finite partition \(\{\tau_n^i=m,\ \iota_m^i=\ell\}\), gives
\(
\widetilde U_{\tau_n^i}^{\iota_{\tau_n^i}^i,\,j^{\iota_{\tau_n^i}^i,\tau_n^i}}(\overline M)
=\overline Y_{t_{\tau_n^i}}^{\iota_{\tau_n^i}^i}
\; \text{on }\{\tau_n^i>n\}.
\)
Hence, by \eqref{eq:stopped-Y}, 
\[
\widetilde U_n^{i,j^{i,n}}(\overline M)
=
\int_{t_n}^{t_{\tau_n^i}} f^i(s)\,ds
+\mathcal R_{\tau_n^i}^{i,\iota_{\tau_n^i}^i}1_{(\tau_n^i<N)}
+\Phi^i1_{(\tau_n^i=N)}
-\overline M_{t_{\tau_n^i}}^i+\overline M_{t_n}^i
=\overline Y_{t_n}^i .
\]

\emph{Case 2: \(\tau_n^i=n\).}
By \Cref{thm:construction-optimal-mode-decision}(iii),
\(
j^{i,n}=j^{\iota_n^i,n}
\; \text{on }\{\tau_n^i=n\}.
\)
Also, \(\mathbb P(D^{i,n})=0\) implies
\(
\tau_n^{\iota_n^i}>n
\; \text{on }\{\tau_n^i=n\},\ \mathbb P\text{-a.s.}
\)
Therefore, applying Case 1 to the pair \((n,\iota_n^i)\),
\[
\widetilde U_n^{\iota_n^i,j^{i,n}}(\overline M)
=\widetilde U_n^{\iota_n^i,j^{\iota_n^i,n}}(\overline M)
=\overline Y_{t_n}^{\iota_n^i}
\qquad\text{on }\{\tau_n^i=n\}.
\]
Since \(j_n^{i,n}=\iota_n^i\), the definition of \(\widetilde U\) yields
\[
\widetilde U_n^{i,j^{i,n}}(\overline M)
=\widetilde U_n^{\iota_n^i,j^{i,n}}(\overline M)-l_{i\,\iota_n^i}(t_n)
=\overline Y_{t_n}^{\iota_n^i}-l_{i\,\iota_n^i}(t_n)
=\mathcal R_n^{i,\iota_n^i}
=\overline Y_{t_n}^i
\]
on \(\{\tau_n^i=n\}\). This completes the induction.
\end{proof}

\vspace{0.1em}
\begin{proof}[Proof of Lemma~\ref{lem:switching-times-no-free-lunch}]
The case \(n=N\) is trivial. Fix \(n<N\).

(i). Set \(A:=\{\mathcal J_S=\varnothing\}\). Then, on \(A\), for every \(j\in\mathcal J\),
\(
q:=j_n^{j,n}\neq j,\; r:=j_n^{q,n}\neq q.
\)
Since \(q\neq j\), Property~(i) 
in \Cref{thm:construction-optimal-mode-decision} 
yields \( \tau^j_n = n \) and
\(
\overline Y_{t_n}^j = \mathcal R_n^{j,q}.
\)
Likewise, \(r\neq q\) implies \(\overline Y_{t_n}^q=\mathcal R_n^{q,r}\), hence, by the triangular condition,
\[
\overline Y_{t_n}^j
= \mathcal R_n^{j,q}
=\overline Y_{t_n}^r-l_{qr}(t_n)-l_{jq}(t_n)
<\overline Y_{t_n}^r-l_{jr}(t_n)
=\mathcal R_n^{j,r}
\le \overline Y_{t_n}^j  \quad  \P \text{-a.s}, 
\]
which can not hold with positive probability. Thus \(\mathcal J_S\neq\varnothing \; \; \P\)-a.s.

(ii). Fix \(j\in\mathcal J\), and let \(q:=j_n^{j,n}\). If \(q\notin\mathcal J_S\), then 
\(
r:=j_n^{q,n}\neq q.
\)
As above,
\[
\overline Y_{t_n}^j
=\mathcal R_n^{j,q}
=\overline Y_{t_n}^r-l_{qr}(t_n)-l_{jq}(t_n)
<\mathcal R_n^{j,r}
\le \overline Y_{t_n}^j \quad \P\text{-a.s.} ,
\]
which can not hold with positive probability. Hence \(q\in\mathcal J_S \; \; \P \)-a.s.
\end{proof}

\vspace{0.1em}
\begin{proof}[Proof of Corollary~\ref{cor:max-no-free-lunch}]
The case \(n=N\) is trivial. Fix \(n<N\). For \(j\in\mathcal J_N\), let \(q:=j_n^{j,n}\). By \Cref{lem:switching-times-no-free-lunch}(ii), \(q\in\mathcal J_S \; \; \P\)-a.s. Moreover, by the triangular condition, 
\[
\mathcal R_n^{i,j}
=\overline Y_{t_n}^q-l_{jq}(t_n)-l_{ij}(t_n)
\le \overline Y_{t_n}^q-l_{iq}(t_n)
=\mathcal R_n^{i,q} \quad \P\text{-a.s.}  ,
\]
Hence, 
\(
\max_{j\in\mathcal J}\mathcal R_n^{i,j}
=
\max_{j\in\mathcal J_S}\mathcal R_n^{i,j},
\)
and the claim follows from \(\overline Y_{t_n}^i=\max_{j\in\mathcal J}\mathcal R_n^{i,j}\).
\end{proof}

\vspace{0.1em}
\begin{proof}[Proof of Theorem~\ref{thm:DPP-upper-surely-optimal}]

For the first equality in \eqref{eq:first-second-equality-Y}, split the surely expansion at \(n+1\). By \Cref{thm:construction-optimal-mode-decision}(ii),
\(
j_k^{i,n}=j_k^{\,j_n^{i,n},\,n+1},\; k\in\overline N_{n+1},
\)
hence
\[
\overline Y_{t_n}^i
=
\int_{t_n}^{t_{n+1}} f^{j_n^{i,n}}(s)\,ds
-l_{i\,j_n^{i,n}}(t_n)
-\Delta\overline M_{t_n}^{\,j_n^{i,n}}
+\overline Y_{t_{n+1}}^{\,j_n^{i,n}}.
\]

For the second equality, define
\(
V_n^j:=
\int_{t_n}^{t_{n+1}} f^j(s)\,ds
-l_{ij}(t_n)
-\Delta\overline M_{t_n}^{\,j}
+\overline Y_{t_{n+1}}^{\,j},
\; j\in\mathcal J.
\)
If \(j\in\mathcal J_S\), then \(j_n^{j,n}=j\), so the first equality applied with initial regime \(j\) gives
\(
V_n^j
=
\overline Y_{t_n}^j-l_{ij}(t_n)
=
\mathcal R_n^{i,j}.
\)
Thus, by \Cref{cor:max-no-free-lunch},
\(
\overline Y_{t_n}^i
=
\max_{j\in\mathcal J_S}\mathcal R_n^{i,j}
=
\max_{j\in\mathcal J_S}V_n^j.
\)
If \(j\in\mathcal J_N\), then by \eqref{eq:equivalent-dual} and \eqref{eq:iterative_dual_upper_two_regime} with one-step comparison,
\(
\overline Y_{t_n}^j
\ge
\int_{t_n}^{t_{n+1}} f^j(s)\,ds
-\Delta\overline M_{t_n}^{\,j}
+\overline Y_{t_{n+1}}^{\,j},
\)
so again, 
\[
V_n^j
\le
\overline Y_{t_n}^j-l_{ij}(t_n)
=
\mathcal R_n^{i,j}
\le
\max_{\ell\in\mathcal J_S}\mathcal R_n^{i,\ell}
=
\max_{\ell\in\mathcal J_S}V_n^\ell.
\]
Hence, \(\max_{j\in\mathcal J}V_n^j=\max_{j\in\mathcal J_S}V_n^j=\overline Y_{t_n}^i\), which proves \eqref{eq:first-second-equality-Y}.

Finally, \eqref{eq:equality-dual-upper-bound} follows by splitting the maximum operator of \eqref{eq:upper_bound_operator} according to the first regime choice \(j_n\):
\[
\widetilde U_n^i(M)
=
\max_{j_n\in\mathcal J}
\Big[
\int_{t_n}^{t_{n+1}} f^{j_n}(s)\,ds
-l_{ij_n}(t_n)
-\Delta M_{t_n}^{\,j_n}
+\widetilde U_{n+1}^{j_n}(M)
\Big]
\]
as claims.
\end{proof}

\vspace{0.1em}
\begin{proof}[Proof of Theorem~\ref{thm:stong_duality}]
By \Cref{thm:expansion-surely-optimal},
\(
\overline Y_{t_n}^i=\widetilde U_n^{i,j^{i,n}}(\overline M)\le \widetilde U_n^i(\overline M),
\; \mathbb P\text{-a.s.}
\)
It remains to prove \(\widetilde U_n^i(\overline M)\le \overline Y_{t_n}^i\). We argue by backward induction in \(n\).

The claim is trivial for \(n=N\), since
\(
\widetilde U_N^i(\overline M)=\Phi^i=\overline Y_{t_N}^i.
\)
Fix \(n<N\), and suppose 
\(
\widetilde U_{n+1}^j(\overline M)\le \overline Y_{t_{n+1}}^j,
\; j\in\mathcal J.
\)
Then, by \Cref{thm:DPP-upper-surely-optimal},
\[
\begin{aligned}
\widetilde U_n^i(\overline M)
&=
\max_{j\in\mathcal J}
\Big[
\int_{t_n}^{t_{n+1}} f^j(s)\,ds
-l_{ij}(t_n)
-\Delta\overline M_{t_n}^j
+\widetilde U_{n+1}^j(\overline M)
\Big] \\
&\le
\max_{j\in\mathcal J}
\Big[
\int_{t_n}^{t_{n+1}} f^j(s)\,ds
-l_{ij}(t_n)
-\Delta\overline M_{t_n}^j
+\overline Y_{t_{n+1}}^j
\Big]
=
\overline Y_{t_n}^i.
\end{aligned}
\]
Hence \(\widetilde U_n^i(\overline M)=\overline Y_{t_n}^i\), and therefore
\(
\overline Y_{t_n}^i
=
\widetilde U_n^{i,j^{i,n}}(\overline M)
=
\widetilde U_n^i(\overline M).
\)

Since \(\overline Y_{t_n}^i\) is \(\mathcal F_{t_n}\)-measurable,
\(
\mathbb E_{t_n}\!\big[\widetilde U_n^i(\overline M)\big]
=
\overline Y_{t_n}^i.
\)
Finally, weak duality gives
\(
\overline Y_{t_n}^i
\le
\essinf_{M\in(\mathcal M_{n,N})^J}\mathbb E_{t_n}\!\big[\widetilde U_n^i(M)\big],
\)
while choosing \(M=\overline M\) yields the reverse inequality. This proves
\eqref{eq:surely-optimal-strong-dual}.
\end{proof}

\vspace{0.1em}
\begin{proof}[Proof of Proposition~\ref{pro:primal_DPP}]
The terminal condition is immediate. By \Cref{thm:expansion-surely-optimal},
\(
\overline Y_{t_n}^i=\widetilde U_n^{i,j^{i,n}}(\overline M).
\)
Since \(j^{i,n}\in\mathcal D_n^i\) is adapted and \(\overline M\) are martingales, taking condition expectation eliminates the martingale increments, therefore 
\(
\overline Y_{t_n}^i
=\E_{t_n}\!\big[\widetilde U_n^{i,j^{i,n}}(\overline M)\big]
=L_n^i(j^{i,n}),
\)
which proves \eqref{eq:primal_optimality}.  \eqref{eq:lower_bound_DPP} is a direct one-step reformulation of \(L_n^i(d^n)\).

Set
\(
A_n^{i,j}:=\E_{t_n}\!\big[\int_{t_n}^{t_{n+1}} f^j(s)\,ds+\overline Y_{t_{n+1}}^j-l_{ij}(t_n)\big],\; j\in\mathcal J.
\)
Then, for any \(d^n\in\mathcal D_n^i\),
\(
L_n^i(d^n)
\le \sum_{j=1}^J 1_{\{d_n=j\}}A_n^{i,j}
\le \max_{j\in\mathcal J} A_n^{i,j},
\)
and hence
\(
\overline Y_{t_n}^i
=\esssup_{d^n\in\mathcal D_n^i}L_n^i(d^n)
\le \max_{j\in\mathcal J}A_n^{i,j}.
\)
Conversely, for each \(j\in\mathcal J\), the concatenation
\(
d^n:=(j,j^{j,n+1})
\)
belongs to \(\mathcal D_n^i\), and, by \eqref{eq:lower_bound_DPP} and \eqref{eq:primal_optimality} at time \(n+1\),
\[
A_n^{i,j}
=\E_{t_n}\!\Big[\int_{t_n}^{t_{n+1}} f^j(s)\,ds+L_{n+1}^j(j^{j,n+1})-l_{ij}(t_n)\Big]
=L_n^i(d^n)
\le \overline Y_{t_n}^i.
\]
Maximizing over \(j\) gives the reverse inequality in \eqref{eq:primal_DPP_value_process}.
\end{proof}

\subsection{Proof of results in Section~\ref{sec:DeepSwitchMart}}

\begin{proof}[Proof of Proposition~\ref{pro:L2-to-L1-loss}]
Part (i) is exactly \cite[Proposition~4.9(i)]{ye2025deepmartingale}. For (ii), set
\(
U_K:=\widetilde U_n^i(M^{\theta_i^K;K})-\eta_n^i,
\;
Y:=\overline Y_{t_n}^i-\eta_n^i.
\)
By weak duality,
\(
\E_{t_n}[U_K]\ge Y\ge0,
\)
and therefore
\(
\E|Y|^2\le \E\big|\E_{t_n}[U_K]\big|^2\le \E|U_K|^2.
\)
On the other hand, by the choice of \(\theta_i^K\) and \Cref{coro:tight_upper}(ii),
\[
\E|U_K|^2
\le
\inf_{\theta\in\Theta^{N\times J}}
\E\big|\widetilde U_n^i(M^{\theta;K})-\eta_n^i\big|^2+\varepsilon_K
\to \E|Y|^2 \; \text{as} \; K \uparrow \infty .
\]
Hence \(\E|U_K|^2\to \E|Y|^2\). Since $ a^2 - b^2 - (a-b)^2 = 2b(a-b) \ge 0 $ for any $ a \ge b \ge 0 $, we have 
\(
0\le \E\big|\E_{t_n}[U_K]-Y\big|^2 
\le \E|U_K|^2-\E|Y|^2,
\)
which implies
\(
\E\big|\E_{t_n}[U_K]-Y\big|^2\to0.
\)
Finally,
\(
\E|U_K-Y|^2
=
\E[\Var_n(U_K)]
+\E\big|\E_{t_n}[U_K]-Y\big|^2,
\)
so part (i) yields \(\E|U_K-Y|^2\to0\). This proves the \(L^2\)-convergence, and the convergence of expectations is immediate.
\end{proof}

\begin{proof}[Proof of Lemma~\ref{lem:V_Lip_growth}]
We argue by backward induction on \(n\). The terminal step is \(V_N^{i;d}=\Phi^{i;d}\). Suppose for \(n\in\N^{-1}\),
\(
|V_{n+1}^{j;d}(x)|\le C d^q(1+\|x\|),
\; 
|V_{n+1}^{j;d}(x)-V_{n+1}^{j;d}(y)|\le C d^q\|x-y\|,
\)
for all \(j\in\mathcal J\). Writing \(X^x:=X^{t_n,x;d}\) and \(X^y:=X^{t_n,y;d}\), define
\[
\Gamma_n^{i,j}(x)
:=
\E\!\Big[
\int_{t_n}^{t_{n+1}} f^{j;d}(s,X_s^x)\,ds
+V_{n+1}^{j;d}(X_{t_{n+1}}^x)
-l_{ij}^d(t_n,x)
\Big].
\]
Then \(V_n^{i;d}=\max_{j\in\mathcal J}\Gamma_n^{i,j}\) by \eqref{eq:primal_DPP_value_function}. By \cite[Theorems~A.7, A.10]{ye2025deepmartingale} and \(d\ge3\), there exist \(C_0,q_0>0\), independent of \(d\), such that
\[
\sup_{s\in[t_n,t_{n+1}]}\| X_s^x\|_{L^2}
\le C_0 d^{q_0}(1+\|x\|),
\; 
\sup_{s\in[t_n,t_{n+1}]}\| X_s^x-X_s^y\|_{L^2}
\le C_0 d^{q_0}\|x-y\|.
\]
Combining these estimates with Assumption~\ref{ass:N_0_structural_ass_model_deter}, induction hypothesis and direct estimation techniques, yields, for some \(C',q'>0\) independent of \(d\),
\(
|\Gamma_n^{i,j}(x)|\le C' d^{q'}(1+\|x\|),
\; 
|\Gamma_n^{i,j}(x)-\Gamma_n^{i,j}(y)|\le C' d^{q'}\|x-y\|.
\)
The pointwise maximum preserves both bounds, hence
\(
|V_n^{i;d}(0)|+\Lip V_n^{i;d}\le C' d^{q'}.
\)
By re-choosing constants, we completes the induction.
\end{proof}

\begin{proof}[Proof of Lemma~\ref{lem:Lip_Hol_lem_f_i}]
For any \(\varepsilon>0\), let \(\hat f_{\varepsilon}^{i;d}\) be as in Assumption~\ref{ass:other_new}. Then, 
\[
    \begin{aligned}
        |f^{i;d}(t,x)-f^{i;d}(t,y)|
        & \le 2\varepsilon cd^q(1+\|x\|+\|y\|) 
        +|\hat f_{\varepsilon}^{i;d}(t,x)-\hat f_{\varepsilon}^{i;d}(t,y)| \\
        & \le 2\varepsilon cd^q(1+\|x\|+\|y\|)+cd^q\|x-y\| , \quad \forall \; x,y\in\mathbb R^d .
    \end{aligned}
\]
Since this holds for all \(\varepsilon>0\), we have \(\Lip f^{i;d}(t,\cdot)\le cd^q\). Similarly, for \(s,t\in[0,T]\),
\[
|f^{i;d}(t,x)-f^{i;d}(s,x)|
\le 2\varepsilon cd^q(1+\|x\|)+cd^q(1+\|x\|)\sqrt{|t-s|},
\]
and letting \(\varepsilon\downarrow0\) yields the \(1/2\)-H\"older bound.
\end{proof}

\begin{proof}[Proof of Lemma~\ref{lem:nn_realization_maximum}]
 The binary maximum is realized by
\[
\max(a,b)=A_2 \varrho \big(A_1(a,b)^\top\big),\quad
A_1=\begin{pmatrix}1&-1\\0&1\\0&-1\end{pmatrix},\quad
A_2=\begin{pmatrix}1&1&-1\end{pmatrix},
\]
where $ \varrho$ is the component-wise ReLU activation. 
This costs size \(7\). Repeating this identity and using parallelization as in \cite[Proposition~2.3]{opschoor20} yields the stated bound.
\end{proof}

\vspace{0.1em}
\begin{proof}[Proof of Theorem~\ref{theorem:recursive_express_new}]
We proceed by backward induction. 
Constants \(C,\alpha,\tau > 0 \) below may change from line to line and may depend on \(n\), but independent of \(d,\delta,\varepsilon\).

\vspace{0.3em}
\paragraph{\it Step 1: Terminal time}
For \(n=N\), \(V_N^{i;d}=\Phi^{i;d}\). Let
\(
\bar\varepsilon:=\frac{\varepsilon}{c d^q(1+k_N d^{p_N})},
\;
\hat V_{N,\varepsilon}^{i;d}:=\hat\Phi_{\bar\varepsilon}^{i;d}.
\)
By Assumption~\ref{ass:other_new},
\(
\|\hat V_{N,\varepsilon}^{i;d}-V_N^{i;d}\|_{2,\rho_{N;d}}
\le \bar\varepsilon c d^q\bigl(1+\mathbb M_{\bar p}(\rho_{N;d})\bigr)
\le \varepsilon.
\)
Moreover,
\(
\size(\hat V_{N,\varepsilon}^{i;d})+\Growth(\hat V_{N,\varepsilon}^{i;d})
\le C d^\alpha \varepsilon^{-r}.
\)
Thus the claim holds at \(n=N\).

\vspace{0.3em}
\paragraph{\it Step 2: Induction hypothesis and continuation value}
Suppose that the statement is true at time \(n+1\), and fix a probability measure \(\rho_{n;d}\) with
\(
\mathbb M_{\bar p}(\rho_{n;d})\le k_n d^{p_n}.
\)
Define the push-forward measure
\(
\hat\rho_{n+1;d}:=(\rho_{n;d}\otimes\mathbb P)\circ (P_{t_n}^{t_{n+1};d})^{-1}.
\)
By Assumption~\ref{ass:dynamic_ass},
\[
\mathbb M_{\bar p}(\hat\rho_{n+1;d})
\le \|\Growth(P_{t_n}^{t_{n+1};d}(*,\cdot))\|_{L^p}
\bigl(1+\mathbb M_{\bar p}(\rho_{n;d})\bigr)
\le k_{n+1} d^{p_{n+1}}  .
\]
Hence, for every \(j\in\mathcal J\) and \(\delta\in(0,1]\), the induction hypothesis yields a deep ReLU network \(\hat V_{n+1,\delta}^{j;d}\) such that
\[
\|\hat V_{n+1,\delta}^{j;d}-V_{n+1}^{j;d}\|_{2,\hat\rho_{n+1;d}}\le \delta,
\quad
\size(\hat V_{n+1,\delta}^{j;d})+\Growth(\hat V_{n+1,\delta}^{j;d})
\le C d^\alpha \delta^{-\tau}.
\]

Let
\(
C_n^{j;d}(x):=\E\big[V_{n+1}^{j;d}(X_{t_{n+1}}^{t_n,x;d})\big].
\)
As in \cite[Theorem~3]{ye2025deepmartingale}, consider
\(
\Gamma_{n,\delta,L}^{j;d}(x)
:=
\frac1L\sum_{l=1}^{L}
\hat V_{n+1,\delta}^{j;d}\!\big(P_{t_n}^{t_{n+1},l;d}(x,\cdot)\big),
\)
where \(P_{t_n}^{t_{n+1},l;d}\), \(l=1,\dots,L\), are i.i.d.\ copies of \(P_{t_n}^{t_{n+1};d}\). By similar estimation techniques in \cite{Jentzen23,gonon23},
\(
\E\big\|\Gamma_{n,\delta,L}^{j;d}
-\E\big[\hat V_{n+1,\delta}^{j;d}(X_{t_{n+1}}^{t_n,*;d})\big]\big\|_{2,\rho_{n;d}}
\le C d^\alpha \delta^{-\tau}L^{-1/2}.
\)
Choose
\(
L_\delta:= \lceil C d^\alpha \delta^{-2-2\tau} \rceil.
\)
By \cite[Proposition~4.14, Lemma~4.25]{ye2025deepmartingale}, we fix \(\omega_0\in\Omega\) such that
\[
\big\|\Gamma_{n,\delta,L_\delta}^{j;d}(\cdot,\omega_0)
-\E\big[\hat V_{n+1,\delta}^{j;d}(X_{t_{n+1}}^{t_n,*;d})\big]\big\|_{2,\rho_{n;d}}
\le C d^\alpha \delta,
\]
and simultaneously all sampled flow realizations have size and growth bounded by \(C d^\alpha \delta^{-\tau}\). By the composition and summation results of \cite{opschoor20,Gonon-Schwab2021-express,Jentzen23}, the map
\(
\gamma_{n,\delta}^{j;d}(x):=\Gamma_{n,\delta,L_\delta}^{j;d}(x,\omega_0)
\)
is a deep ReLU network satisfying
\(
\big\|\gamma_{n,\delta}^{j;d}
-\E\big[\hat V_{n+1,\delta}^{j;d}(X_{t_{n+1}}^{t_n,*;d})\big]\big\|_{2,\rho_{n;d}}
\le C d^\alpha \delta
\)
and
\(
\size(\gamma_{n,\delta}^{j;d})+\Growth(\gamma_{n,\delta}^{j;d})
\le C d^\alpha \delta^{-\tau}.
\)
Combining this with
\[
\big\|\E\big[\hat V_{n+1,\delta}^{j;d}(X_{t_{n+1}}^{t_n,*;d})\big]
-\E\big[V_{n+1}^{j;d}(X_{t_{n+1}}^{t_n,*;d})\big]\big\|_{2,\rho_{n;d}}
\le
\|\hat V_{n+1,\delta}^{j;d}-V_{n+1}^{j;d}\|_{2,\hat\rho_{n+1;d}}
\le \delta,
\]
we obtain
\(
\|\gamma_{n,\delta}^{j;d}-C_n^{j;d}\|_{2,\rho_{n;d}}
\le C d^\alpha \delta .
\)

\vspace{0.3em}
\paragraph{\it Step 3: Running payoff -- quadrature deep ReLU realization}
This is an additional tricky part compared with the continuation value approximation in \cite{gonon23,ye2025deepmartingale}. We first discretize the time integral, and then realize the resulting quadrature--Monte Carlo approximation by a deterministic deep ReLU network.

For \(j\in\mathcal J\), define
\(
R_n^{j;d}(x):=\E\!\big[\int_{t_n}^{t_{n+1}} f^{j;d}(s,X_s^{t_n,x;d})\,ds\big].
\)
For any \(B\in\mathbb N_+\), let \(\Delta s=(t_{n+1}-t_n)/B\) and \(s_b=t_n+b\Delta s\), \(b=0,\dots,B-1\), and define
\(
Q_{n,B}^{j;d}(x):=\sum_{b=0}^{B-1}\E\big[f^{j;d}(s_b,X_{s_b}^{t_n,x;d})\big]\Delta s.
\)
For \(s\in[s_b,s_{b+1}]\),
\[
\begin{aligned}
&\|f^{j;d}(s,X_s^{t_n,x;d})-f^{j;d}(s_b,X_{s_b}^{t_n,x;d})\|_{L^2} \\
&\le
\| f^{j;d}(s,X_s^{t_n,x;d})-f^{j;d}(s_b,X_s^{t_n,x;d})\|_{L^2}
+
\|f^{j;d}(s_b,X_s^{t_n,x;d})-f^{j;d}(s_b,X_{s_b}^{t_n,x;d})\|_{L^2} \\
&\le
(\Holdertwo f^{j;d}) \,(1+\|x\|)\sqrt{|s-s_b|}
+
\Lip f^{j;d}(s_b,\cdot)\,
\|X_s^{t_n,x;d}-X_{s_b}^{t_n,x;d}\|_{L^2} \\
&\le C d^\alpha (1+\|x\|)\sqrt{\Delta s},
\end{aligned}
\]
where we used Lemma~\ref{lem:Lip_Hol_lem_f_i} and Assumption~\ref{ass:dynamic_ass}. Integrating over each subinterval and then taking the \(L^2(\rho_{n;d})\)-norm gives
\(
\|R_n^{j;d}-Q_{n,B}^{j;d}\|_{2,\rho_{n;d}}
\le C d^\alpha B^{-1/2}.
\)
Hence, with
\(
B_\delta:= \lceil C d^\alpha \delta^{-2} \rceil,
\)
we obtain
\(
\|R_n^{j;d}-Q_{n,B_\delta}^{j;d}\|_{2,\rho_{n;d}}\le \delta .
\)

Next, approximate \(f^{j;d}\) by the network \(\hat f_\delta^{j;d}\) from Assumption~\ref{ass:other_new} and define
\(
\hat Q_{n,\delta}^{j;d}(x)
:=
\sum_{b=0}^{B_\delta-1}
\E\big[\hat f_\delta^{j;d}(s_b,P_{t_n}^{s_b;d}(x,\cdot))\big]\Delta s.
\)
Since
\(
|\hat f_\delta^{j;d}(s_b,x)-f^{j;d}(s_b,x)|\le \delta c d^q(1+\|x\|),
\)
Assumption~\ref{ass:dynamic_ass} yields
\(
\|\hat Q_{n,\delta}^{j;d}-Q_{n,B_\delta}^{j;d}\|_{2,\rho_{n;d}}
\le C d^\alpha \delta.
\)
Therefore,
\[
\|\hat Q_{n,\delta}^{j;d}-R_n^{j;d}\|_{2,\rho_{n;d}}
\le C d^\alpha \delta.
\]

We now approximate \(\hat Q_{n,\delta}^{j;d}\) by a deep ReLU network. Consider the Monte Carlo approximation
\(
\Lambda_{n,\delta,L}^{j;d}(x)
:=
\frac1L\sum_{l=1}^{L}\sum_{b=0}^{B_\delta-1}
\hat f_\delta^{j;d}\!\big(s_b,P_{t_n}^{s_b,l;d}(x,\cdot)\big)\Delta s,  
\)
where \(P_{t_n}^{s_b,l;d}\), \(l=1,\dots,L\), are i.i.d.\ copies of \(P_{t_n}^{s_b;d}\). Since \(\sum_{b=0}^{B_\delta-1}\Delta s=t_{n+1}-t_n\), the growth bound in Assumption~\ref{ass:other_new} 
imply
\[
\Big(
\int_{\mathbb R^d}
\E\Big|
\sum_{b=0}^{B_\delta-1}
\hat f_\delta^{j;d}(s_b,X_{s_b}^{t_n,x;d})\Delta s
\Big|^2 \rho_{n;d}(dx)
\Big)^{\frac{1}{2}}
\le C d^\alpha \delta^{-r}.
\]
Hence, by estimation techinques in \cite{Jentzen23,gonon23},
\(
\E\|\Lambda_{n,\delta,L}^{j;d}-\hat Q_{n,\delta}^{j;d}\|_{2,\rho_{n;d}}
\le C d^\alpha \delta^{-r}L^{-1/2}.
\)
Choose
\(
L_\delta^0:= \lceil C d^\alpha \delta^{-2-2r}(B_\delta+1)^2 \rceil.
\)
Then
\(
\E\|\Lambda_{n,\delta,L_\delta^0}^{j;d}-\hat Q_{n,\delta}^{j;d}\|_{2,\rho_{n;d}}
\le \frac{\delta}{B_\delta+1}.
\)

Moreover, by Assumption~\ref{ass:dynamic_ass},
\(
\E\big[\size(P_{t_n}^{s_b,l;d}(*,\cdot))\big]
+\E\big[\Growth(P_{t_n}^{s_b,l;d}(*,\cdot))\big]
\le c d^q
\)
uniformly in \(b,l\). Thus, \cite[Lemma~4.25]{ye2025deepmartingale} yields an \(\omega_0\in\Omega\) such that
\[
\|\Lambda_{n,\delta,L_\delta^0}^{j;d}(\cdot,\omega_0)-\hat Q_{n,\delta}^{j;d}\|_{2,\rho_{n;d}}
\le C d^{\alpha} \delta,
\]
and simultaneously
\[
\max_{\substack{0\le b\le B_\delta-1\\1\le l\le L_\delta^0}}
\big\{
\size(P_{t_n}^{s_b,l;d}(*,\omega_0))
+
\Growth(P_{t_n}^{s_b,l;d}(*,\omega_0))
\big\}
\le C d^\alpha \delta^{-\tau}.
\]

For each \(b\), \(\hat f_\delta^{j;d}\) yields a deep ReLU network
\(
\hat f_{b,\delta}^{j;d}(x):=\hat f_\delta^{j;d}(s_b,x),
\;
\size(\hat f_{b,\delta}^{j;d})\le \size(\hat f_\delta^{j;d}),
\)
see \cite[Lemma~4.9]{gonon23}. Therefore, by \cite[Proposition~2.2]{opschoor20}, \cite[Lemma~3.2]{Gonon-Schwab2021-express}, the deterministic map
\(
\lambda_{n,\delta}^{j;d}(x):=\Lambda_{n,\delta,L_\delta^0}^{j;d}(x,\omega_0)
\)
is a deep ReLU network. Moreover,
\(
\|\lambda_{n,\delta}^{j;d}-\hat Q_{n,\delta}^{j;d}\|_{2,\rho_{n;d}}\le C d^{\alpha} \delta,
\)
and after summing up the bounds over all \(b\) and \(l\),
\[
\size(\lambda_{n,\delta}^{j;d})+\Growth(\lambda_{n,\delta}^{j;d})
\le C d^\alpha \delta^{-\tau}.
\]
Combining with the bound for \(\hat Q_{n,\delta}^{j;d}-R_n^{j;d}\), we obtain
\(
\|\lambda_{n,\delta}^{j;d}-R_n^{j;d}\|_{2,\rho_{n;d}}
\le C d^\alpha \delta.
\)

\paragraph{\it Step 4: Assemble into maximum operator}
By \eqref{eq:primal_DPP_value_function}, for \(i,j\in\mathcal J\), we define
\[
F_n^{ij;d}(x):=
R_n^{j;d}(x)+C_n^{j;d}(x)-l_{ij}^d(t_n,x),
\quad
V_n^{i;d}(x)=\max_{j\in\mathcal J}F_n^{ij;d}(x)
\]
 Let \(\hat l_{ij,\delta}^{n;d}\) be the deep ReLU approximation of \(l_{ij}^d(t_n,\cdot)\) from Assumption~\ref{ass:other_new}, and set
\[
\varphi_{n,\delta}^{ij;d}
:=
\lambda_{n,\delta}^{j;d}
+\gamma_{n,\delta}^{j;d}
-\hat l_{ij,\delta}^{n;d},
\quad
\hat V_{n,\varepsilon}^{i;d}
:=
\max_{j\in\mathcal J}\varphi_{n,\delta}^{ij;d}.
\]
From the previous estimates and Assumption~\ref{ass:other_new},
\(
\|\varphi_{n,\delta}^{ij;d}-F_n^{ij;d}\|_{2,\rho_{n;d}}
\le C d^\alpha \delta
\)
uniformly in $i,j$.
Using
\(
|\max_j a_j-\max_j b_j | \le \max_j |a_j-b_j|,
\)
we have
\(
\|\hat V_{n,\varepsilon}^{i;d}-V_n^{i;d}\|_{2,\rho_{n;d}}
\le C d^\alpha \delta.
\)
Choose
\(
\delta:=\frac{\varepsilon}{C d^\alpha}\in(0,1]
\)
(enlarging \(C,\alpha\) if necessary for $C d^\alpha \ge 1$). 
Then
\[
\|\hat V_{n,\varepsilon}^{i;d}-V_n^{i;d}\|_{2,\rho_{n;d}}\le \varepsilon.
\]

Finally, the closure properties of ReLU networks under composition, finite sums
\cite{opschoor20,Gonon-Schwab2021-express}, together with Lemma~\ref{lem:nn_realization_maximum}, imply
\[
\size(\hat V_{n,\varepsilon}^{i;d})+\Growth(\hat V_{n,\varepsilon}^{i;d})
\le c_n d^{q_n}\varepsilon^{-\tau_n}
\]
for suitable constants \(c_n,q_n,\tau_n\). This completes the induction.
\end{proof}

\begin{proof}[Proof of Theorem~\ref{thm:express_deep_mtg}]
Let \(K_{\varepsilon;d}\) be given by \Cref{thm:express_N_0_new} with accuracy \(\varepsilon^2/4\); then
\(
K_{\varepsilon;d}\le {b^{*}}d^{q^{*}}\varepsilon^{-2} / 4 ,
\)
and for all \(m\in\N^{-1}\), \(j\in\mathcal J\),
\[
\E\!\Big[\sum_{k=0}^{K_{\varepsilon;d}-1}\int_{t_k^m}^{t_{k+1}^m}
\|\overline Z_s^{j;d}-\hat Z_{t_k^m}^{j;K_{\varepsilon;d},d}\|^2\,ds\Big]
\le \frac{\varepsilon^2}{4}.
\]
Next, apply \Cref{theorem:joint nerual network realization_new} with \(K=K_{\varepsilon;d}\) and accuracy \(\varepsilon/2\) to obtain networks \(\tilde z_{m,\varepsilon}^{j;d}:=\tilde z_{m,\varepsilon/2}^{j;K_{\varepsilon;d},d}\) such that
\(
\|\tilde z_{m,\varepsilon}^{j;d}-z_m^{j;K_{\varepsilon;d},d}\|_{2,\mu_m^{K_{\varepsilon;d};d}} 
\le \frac{\varepsilon}{2},
\)
and
\(
\size(\tilde z_{m,\varepsilon}^{j;d})
+\Growth(\tilde z_{m,\varepsilon}^{j;d}(t,\cdot))
\le C d^{Q}\varepsilon^{-R},
\; t\in[t_m,t_{m+1}),
\)
after absorbing the factor \(K_{\varepsilon;d}^{\bar m_m}\) into the exponents.

Using \eqref{eq_Z}, and since \(\overline Y_{t_n}^{i;d}=\widetilde U_n^{i;d}(\overline M^d)\), the same estimate for \Cref{lem:error-propagate} gives
\[
\begin{aligned}
& \Big(\E\big[\max_{i\in\mathcal J}
|\widetilde U_n^{i;d}(\widetilde M_\varepsilon^d)-\overline Y_{t_n}^{i;d}|^2\big]\Big)^{\frac{1}{2}} \le
\sum_{j=1}^J\sum_{m=n}^{N-1}
\|\tilde z_{m,\varepsilon}^{j;d}-z_m^{j;K_{\varepsilon;d},d}\|_{2,\mu_m^{K_{\varepsilon;d};d}} \\
& \qquad \qquad \qquad+
\sum_{j=1}^J\sum_{m=n}^{N-1}
\Big(\E\!\Big[\sum_{k=0}^{K_{\varepsilon;d}-1}\int_{t_k^m}^{t_{k+1}^m}
\|\hat Z_{t_k^m}^{j;K_{\varepsilon;d},d}-\overline Z_s^{j;d}\|^2\,ds\Big]\Big)^{\frac{1}{2}}\\
& \le
\frac12 J(N-n)\varepsilon+\frac12 J(N-n)\varepsilon
=
J(N-n)\varepsilon.
\end{aligned}
\]
This proves the approximation bound, and the expressivity bound follows after maximizing over \(n\in\N^{-1}\).
\end{proof}

\begin{proof}[Proof of Proposition~\ref{lem:delta_deephedge}]
We have
\(
\widetilde u_n^{i;d}(t,X_t^{t_n,x;d})
=
\E\!\big[V_{n+1}^{i;d}(X_{t_{n+1}}^{t_n,x;d})\,|\,\mathcal F_t\big]
\)
by the Markov property, 
hence \(\widetilde u_n^{i;d}(t,X_t^{t_n,x;d})\) is a martingale on \([t_n,t_{n+1}]\). Moreover, 
\[
d\widetilde u_n^{i;d}(t,X_t^{t_n,x;d})
=
(\partial_t\widetilde u_n^{i;d}+\mathcal L^d\widetilde u_n^{i;d})(t,X_t^{t_n,x;d})\,dt
+
(\nabla_x\widetilde u_n^{i;d}\,\sigma^d)(t,X_t^{t_n,x;d})\cdot dW_t^d
\]
by It\^o's formula, 
where \(\mathcal L^d\) is the generator of \(X^{d}\). Since the left-hand side is a martingale, the drift vanishes, and the martingale integrand is \(\overline Z^{i;d}\). The identity for \(\Pi_t^{i,n;d}\) is immediate when \(\sigma^d(t,x)\) is invertible.
\end{proof}



\clearpage
\section*{Supplementary Material}

\section{Supplementary results for iterative stopping problem and its duality}\label{sec:iterative-stopping-supplement}
The reduction of an optimal switching problem to an iterated optimal stopping formulation is well established in the literature. In continuous time, we refer to \cite{Hamadene-Djehiche-09,Martyr-16-signed-switching}. Here we state the corresponding formulation in discrete time, following \cite[Theorem~3.1]{martyr16-discrete-switching}.

\begin{lemma}[Equivalence to iterative optimal stopping]\label{lem:equivalent-primal-iterative-stopping-supplement}
For any $ i\in\mathcal{J} $, $ n\in \overline{N} $,
\begin{equation}\label{eq:equivalent-primal} 
\overline{Y}^{i}_{t_n}
= \esssup_{\tau\in\mathcal{T}_{n}}
\mathbb{E}_{t_n}\Big[
\int_{t_n}^{t_\tau} f^i(s) ds
+ \overline{\mathcal{R}}^{i}_{\tau} 1_{(\tau<N)}
+ \Phi^{i}1_{(\tau=N)}
\Big].
\end{equation}
\end{lemma}

By the Snell envelope results in \cite[Proposition~3.1, Lemma~A.1]{martyr16-discrete-switching}, it follows from \eqref{eq:equivalent-primal} that for all \(i\in\mathcal{J}\) and \(n\in \overline{N}\),
\begin{equation}\label{eq:snell-envelope-shifted}
\begin{aligned}
  &  \overline{Y}^{i}_{t_n} + \int_{0}^{t_n} f^i(s) ds  = \esssup_{\tau\in\mathcal{T}_{n}}
\mathbb{E}_{t_n}\Big[
\int_{0}^{t_\tau} f^i(s) ds
+ \overline{\mathcal{R}}^{i}_{\tau} 1_{(\tau<N)}
+ \Phi^{i}1_{(\tau=N)}
\Big].
\end{aligned}
\end{equation}
Consequently, we obtain the following supermartingale domination property.

\begin{lemma}\label{lem:supermartingale_envelop}
For each \(i\in\mathcal{J}\), the process
\[
\Big(\overline{Y}^{i}_{t_n} + \int_{0}^{t_n} f^i(s) ds\Big)_{n=0}^{N}
\]
is the smallest supermartingale dominating
\[
\Big(\int_{0}^{t_n} f^i(s) ds 
+ \overline{\mathcal{R}}^{i}_{n} 1_{(n<N)}
+ \Phi^{i}1_{(n=N)}\Big)_{n=0}^{N}.
\]
In particular, for \(n\in \overline{N}\),
\begin{equation}\label{eq:supermartingale-dominate}
\overline{Y}^{\,i}_{t_n}
\ge \overline{\mathcal{R}}^{i}_{n}  1_{(n<N)}
+ \Phi^{i}1_{(n=N)}.
\end{equation}
Moreover, for any discrete stopping time \(\tau\in\mathcal{T}^N\),
\begin{equation}\label{eq:supermartingale-dominate-stopping}
\overline{Y}^{i}_{t_\tau}
\ge  \overline{\mathcal{R}}^{i}_{\tau} 1_{(\tau<N)}
+ \Phi^{i}1_{(\tau=N)}.
\end{equation}
Finally,
\begin{equation}\label{eq:tau-star-i-n}
\tau^{*,i}_n
:= \inf\Big\{m\in \overline{N}_n :
\overline{Y}^{i}_{t_m}
= \overline{\mathcal{R}}^{i}_{m} 1_{(m<N)}
+ \Phi^{i}1_{(m=N)}
\Big\}
\end{equation}
is an optimal stopping time for \eqref{eq:equivalent-primal}.
\end{lemma}

\section{Supplementary measurability results for stopping time $ \overline{m}(n,i) $ and regime process $ j(n,i) $}

\begin{lemma}\label{lem:stopping-time}
For any \(i\in\mathcal{J}\) and \(n\in\overline{N}\), the random time \(\overline{m}(n,i)\) is an \(\mathbb{F}_n\)-stopping time.
\end{lemma}

\begin{proof}
Fix \(i\in\mathcal{J}\) and \(n\in\overline{N}\). For any \(m\in\overline{N}_n\) and \(k=n,\ldots,m-1\), we have
\begin{equation}\label{eq:dual-stopping-set-equivalent}
\overline{m}(n,i)=m
\;\Leftrightarrow\;
\overline{Y}^{i}_{t_n}=\overline{U}^{i}_{n,m}(\overline{M}^i)
\ \text{and}\
\overline{Y}^{i}_{t_n}>\overline{U}^{i}_{n,k}(\overline{M}^i)\ \text{for all }k=n,\ldots,m-1.
\end{equation}
The event on the right-hand side of \eqref{eq:dual-stopping-set-equivalent} is \(\mathcal{F}_{t_m}\)-measurable. Hence \(\{\overline{m}(n,i)=m\}\in\mathcal{F}_{t_m}\) for every \(m\in\overline{N}_n\), which proves that \(\overline{m}(n,i)\) is a stopping time.
\end{proof}

\begin{lemma}[Dynamic Programming principle and optimality of \(\overline{m}(n,i)\)]\label{lem:DPP-dual-stopping-optimal}
For any \(n\in \overline{N} \) and \(i\in\mathcal{J}\), the stopping times \(\overline{m}(n,i)\) satisfy the dynamic programming identity
\[
\overline{m}(n,i)=n\,1_{(\overline{m}(n,i)=n)}+\overline{m}(n+1,i)\,1_{(\overline{m}(n,i)>n)}.
\]
Moreover, \(\overline{m}(n,i)\) is optimal for \eqref{eq:equivalent-primal} and admits the following representation \(\mathbb{P}\)-a.s.:
\begin{equation}\label{eq:dual-stopping-representation}
\overline{m}(n,i)
= \inf\Big\{m\in \overline{N}_n :
\overline{Y}^{i}_{t_m}
= \overline{\mathcal{R}}^{i}_{m} 1_{(m<N)}+\Phi^{i}1_{(m=N)}
\Big\} ,
\end{equation}
and, in particular,  
\begin{equation}\label{eq:stopping-equality-dual}
\overline{Y}^{i}_{t_{\overline{m}(n,i)}}
= \overline{\mathcal{R}}^{i}_{\overline{m}(n,i)}  1_{(\overline{m}(n,i)<N)}+\Phi^{i}1_{(\overline{m}(n,i)=N)}.
\end{equation}
\end{lemma}


\begin{proof}
Fix \(i\in\mathcal{J}\) and \(n\in\overline{N}\). On the event \(\{\overline{m}(n,i)>n\}\), the maximizer in \eqref{eq:equivalent-dual} is attained strictly after time \(t_n\), which implies
\(
\overline{Y}^{i}_{t_n}
> \overline{\mathcal{R}}^{i}_{n} 1_{(n<N)}.
\)
Using \eqref{eq:equivalent-dual} and separating the first step from \(t_n\) to \(t_{n+1}\), we obtain on \(\{\overline{m}(n,i)>n\}\),
\begin{align*}
\overline{Y}^{i}_{t_n}
= &\ \overline{U}^i_{n+1}(\overline{M}^i)
+ \int_{t_n}^{t_{n+1}} f^i(s)ds
+ \overline{M}^{i}_{t_n}-\overline{M}^{i}_{t_{n+1}} \\
= &\ \overline{Y}^{i}_{t_{n+1}}
+ \int_{t_n}^{t_{n+1}} f^i(s)ds
+ \overline{M}^{i}_{t_n}-\overline{M}^{i}_{t_{n+1}}.
\end{align*}
Consequently, on \(\{\overline{m}(n,i)>n\}\),
\begin{align*}
\overline{m}(n,i)
= &\ \inf\Big(\argmax_{m=n+1,\ldots,N} \overline{U}^{i}_{n+1,m}(\overline{M}^i)\Big)
= \overline{m}(n+1,i),
\end{align*}
which yields the stated DPP:
\(
\overline{m}(n,i)=n\,1_{(\overline{m}(n,i)=n)}+\overline{m}(n+1,i)\,1_{(\overline{m}(n,i)>n)}.
\)

Next, note that
\(
\overline{m}(n,i)=n
\ \Leftrightarrow\
\overline{Y}^{i}_{t_n}
= \overline{\mathcal{R}}^{i}_{n} 1_{(n<N)}+\Phi^{i}1_{(n=N)}.
\)
Assuming by backward induction that \eqref{eq:dual-stopping-representation} holds for \(\overline{m}(n+1,i)\), the DPP implies
\[
\overline{m}(n,i)
= \inf\Big\{m\in \overline{N}_n :
\overline{Y}^{i}_{t_m}
= \overline{\mathcal{R}}^{i}_{m} 1_{(m<N)}+\Phi^{i}1_{(m=N)}
\Big\},
\]
which proves \eqref{eq:dual-stopping-representation}. Optimality of \(\overline{m}(n,i)\) for \eqref{eq:equivalent-primal} follows from Lemma~\ref{lem:supermartingale_envelop}. Finally, \eqref{eq:stopping-equality-dual} is an immediate consequence of \eqref{eq:dual-stopping-representation}.
\end{proof}

\begin{lemma}\label{lem:measurability_mode_indicator_distinguish}
For any \(i\in\mathcal{J}\) and \(n \in \overline{N} \), the random variable \(j(n,i)\) is \(\mathcal{F}_{t_n}\)-measurable.
\end{lemma}

\begin{proof}
The claim is immediate for \(n=N\) since \(j(N,i)=i\) by definition. Let \(n<N\). Then \(j(n,i)\in\mathcal{J}^{-i}\). Fix any \(j\in\mathcal{J}^{-i}\). 
By direct verification, the event \(\{j(n,i)=j\}\) can be written as
\[
\begin{aligned}
\{j(n,i)=j\}
= &\ \Big(\bigcap_{j>k\in\mathcal{J}^{-i}}
\big\{\overline{Y}^{k}_{t_n}-l_{ik}(t_n)<\overline{Y}^{j}_{t_n}-l_{ij}(t_n)\big\}\Big) \\
&\ \cap
\Big(\bigcap_{j<k\in\mathcal{J}^{-i}}
\big\{\overline{Y}^{k}_{t_n}-l_{ik}(t_n)\le \overline{Y}^{j}_{t_n}-l_{ij}(t_n)\big\}\Big).
\end{aligned}
\]
Since \(\overline{Y}^{k}_{t_n}-l_{ik}(t_n)\in\mathcal{F}_{t_n}\) for every \(k\in\mathcal{J}^{-i}\), each set in the intersections is \(\mathcal{F}_{t_n}\)-measurable, and hence \(\{j(n,i)=j\}\in\mathcal{F}_{t_n}\). Therefore \(j(n,i)\) is \(\mathcal{F}_{t_n}\)-measurable.
\end{proof}

\section{Supplementary results for affine It\^o diffusion}\label{subsec:supplement-affine-ito-diffusion}

For an affine It\^o diffusion \( X^{d} \) (\Cref{def:AID}) under Assumption~\ref{def:ADI_express}, we can establish the following additional H\"older-type continuity property of \( X^{d} \), which is implied in \cite[Proof of Theorem~3.9]{ye2025deepmartingale}.

\begin{lemma}[Expressivity of H\"older continuity]\label{lem:express_holder_AID}
    Suppose \( X^{d} \) satisfies Assumption~\ref{def:ADI_express}. Then there exist constants \( c,q > 0 \) independent of \( d \), such that
    \[
        \mathbb{E}\big[\| X_{t}^{t_n,x;d} - X_{s}^{t_n,x;d} \|\big] \le c d^q (1+\|x\|)|t-s|^{\frac{1}{2}},
    \]
    for all \( t,s \in [t_n,t_{n+1}] \), \( x \in \mathbb{R}^d \), and \( n \in \N^{-1} \).
\end{lemma}

\begin{proof}
By the same argument as in \cite[Lemma~6]{ye2025deepmartingale}, for any \( \tilde{p} \ge 2 \), there exist positive constants \( c_{\tilde{p}}, q_{\tilde{p}} \), independent of \( d \), such that
\[
    \Big(\mathbb{E}\big[ |\Growth(X_t^{s,\cdot;d})|^{\tilde{p}} \big] \Big)^{\frac{1}{\tilde{p}}} \le c_{\tilde{p}} d^{q_{\tilde{p}}},
\]
for any \( s,t \in [0,T] \).
Moreover, using the same techniques as in \cite[Proof of Theorem~2]{ye2025deepmartingale} (in particular, the a priori estimate for \( X^{t_n,x;d} \)), there exist positive constants \( c,q \), independent of \( d \), such that
\[
    \Big(\mathbb{E}\big[\|X_{t}^{t_n,x;d} - X_{s}^{t_n,x;d}\|^2 \big] \Big)^{\frac{1}{2}} \le c d^q (1+\|x\|)|t-s|^{\frac{1}{2}},
\]
for any \( t,s \in [t_n,t_{n+1}] \), \( x \in \mathbb{R}^d \), and \( n \in \N^{-1} \). This yields the desired result.
\end{proof}

Using Lemma~\ref{lem:express_holder_AID} and following the same proof strategy as in \cite[Lemma~6]{ye2025deepmartingale}, we can verify that, under Assumption~\ref{def:ADI_express}, the dynamics structural conditions required by our expressivity framework for DeepMartingales (see Assumption~\ref{ass:N_0_structural_ass_model_deter} and Assumption~\ref{ass:dynamic_ass} in \Cref{subsec:converge_express_ReLU}) are satisfied. We therefore omit the proof.

\begin{lemma}\label{lemma:AID-log_ass1}
    If the affine It\^o diffusion \( X^{d} \) satisfies Assumption~\ref{def:ADI_express}, then \( X^{d} \) satisfies Assumption~\ref{ass:N_0_structural_ass_model_deter} and Assumption~\ref{ass:dynamic_ass} for any \( p > 4 \).
\end{lemma}




\bibliographystyle{siamplain}
\bibliography{references}

@book{SDE03,
author = {{\O}ksendal, Bernt},
title="Stochastic Differential Equations",
year="2003",
publisher="Springer Berlin",
address="Heidelberg"
}

@article{Gonon-Schwab2021-express,
  title={Deep ReLU network expression rates for option prices in high-dimensional, exponential Lévy models},
  author={Gonon, Lukas and Schwab, Christoph},
  journal={Finance Stoch.},
  volume={25},
  number={3},
  pages={615--657},
  year={2021},
  publisher={Springer},
}

@article{Carmona-Ludkoviski01122008,
author = {René Carmona and Michael Ludkovski},
title = {Pricing Asset Scheduling Flexibility using Optimal Switching},
journal = {Appl. Math. Finance},
volume = {15},
number = {5-6},
pages = {405--447},
year = {2008},
publisher = {Routledge},
}

@article{Bernan-Schwartz-85,
 author = {Michael J. Brennan and Eduardo S. Schwartz},
 journal = {J. Bus.},
 number = {2},
 pages = {135--157},
 title = {Evaluating Natural Resource Investments},
 urldate = {2026-01-06},
 volume = {58},
 year = {1985}
}

@article{Dixit-89,
 author = {Avinash Dixit},
 journal = {J. Political Econ.},
 number = {3},
 pages = {620--638},
 title = {Entry and Exit Decisions under Uncertainty},
 urldate = {2026-01-06},
 volume = {97},
 year = {1989}
}

@article{Carmona-Ludkovski-10,
author = {René Carmona and Michael Ludkovski},
title = {Valuation of energy storage: an optimal switching approach},
journal = {Quant. Finance},
volume = {10},
number = {4},
pages = {359--374},
year = {2010},
}

@article{opschoor20,
author = {Opschoor, Joost A. A. and Petersen, Philipp C. and Schwab, Christoph},
title = {Deep {R}e{LU} networks and high-order finite element methods},
journal = {Anal. Appl.},
volume = {18},
number = {05},
pages = {715-770},
year = {2020},
}

@article{Oksendal-Brekke-94,
author = {Brekke, Kjell Arne and Øksendal, Bernt},
title = {Optimal Switching in an Economic Activity under Uncertainty},
journal = {SIAM J. Control Optim.},
volume = {32},
number = {4},
pages = {1021-1036},
year = {1994},
}

@inproceedings{Duckworth-zervos-00,
  author    = {Duckworth, K. and Zervos, M.},
  title     = {A problem of stochastic impulse control with discretionary stopping},
  booktitle = {Proceedings of the 39th IEEE Conference on Decision and Control},
  year      = {2000},
  pages     = {222--227}
}

@article{zervos-98,
author = {Knudsen, Thomas S. and Meister, Bernhard and Zervos, Mihail},
title = {Valuation of Investments in Real Assets with Implications for the Stock Prices},
journal = {SIAM J. Control Optim.},
volume = {36},
number = {6},
pages = {2082-2102},
year = {1998},
}

@article{Pham-switch-07,
author = { Ly Vath, Vathana and  Pham, Huy\^{e}n},
title = {Explicit Solution to an Optimal Switching Problem in the Two‐Regime Case},
journal = {SIAM J. Control Optim.},
volume = {46},
number = {2},
pages = {395-426},
year = {2007},
}

@article{Hamadene-Djehiche-09,
author = {Djehiche, Boualem and Hamad\`{e}ne, Said and Popier, Alexandre},
title = {A Finite Horizon Optimal Multiple Switching Problem},
journal = {SIAM J. Control Optim.},
volume = {48},
number = {4},
pages = {2751-2770},
year = {2009},
}

@article{Hamadene-07,
author = {Hamad\`{e}ne, Said and Jeanblanc, Monique},
title = {On the Starting and Stopping Problem: Application in Reversible Investments},
journal = {Math. Oper. Res.},
volume = {32},
number = {1},
pages = {182-192},
year = {2007},
}

@phdthesis{ludkovski2005-thesis,
  author    = {Ludkovski, Michael},
  title     = {Optimal Switching with Application to Energy Tolling Agreements},
  school    = {Princeton University},
  year      = {2005},
  type      = {{PhD} Thesis}
}

@article{RAISSI2019-pinns,
title = {Physics-informed neural networks: A deep learning framework for solving forward and inverse problems involving nonlinear partial differential equations},
journal = {J. Comput. Phys.},
volume = {378},
pages = {686-707},
year = {2019},
author = {M. Raissi and P. Perdikaris and G.E. Karniadakis},
}

@article{Pham-sifin-14,
author = {A\"{\i}d, Ren\'{e} and Campi, Luciano and Langren\'{e}, Nicolas and Pham, Huy\^{e}n},
title = {A Probabilistic Numerical Method for Optimal Multiple Switching Problems in High Dimension},
journal={SIAM J. Financ. Math.},
volume = {5},
number = {1},
pages = {191-231},
year = {2014},
}

@article{han-weinanE18,
author = { Jiequn Han  and Arnulf Jentzen  and Weinan E },
title = {Solving high-dimensional partial differential equations using deep learning},
journal = {Proc. Natl. Acad. Sci. U. S. A.},
volume = {115},
number = {34},
pages = {8505-8510},
year = {2018},
}

@article{Becker-Jentzen-Neufeld-Deep-Splitting-21,
	author = {Beck, Christian and Becker, Sebastian and Cheridito, Patrick and Jentzen, Arnulf and Neufeld, Ariel},
	title = {Deep Splitting Method for Parabolic {PDEs}},
	journal = {SIAM J. Sci. Comput.},
	volume = {43},
	number = {5},
	pages = {A3135-A3154},
	year = {2021}
}

@article{pham2020deepBSDE,
author = {Huré, Côme and Pham, Huyên and Warin, Xavier},
year = {2020},
month = {01},
pages = {1},
title = {Deep backward schemes for high-dimensional nonlinear PDEs},
volume = {89},
journal = {Math. Comput.},
}

@article{pham2022deepBSDE_erroranalysis,
author = {Germain, Maximilien and Pham, Huy\^{e}n and Warin, Xavier},
title = {Approximation Error Analysis of Some Deep Backward Schemes for Nonlinear PDEs},
journal = {SIAM J. Sci. Comput.},
volume = {44},
number = {1},
pages = {A28-A56},
year = {2022},
}

@article{Jentzen23,
author = {Grohs, P. and Hornung, F. and Jentzen, A. and et al.},
title = {A Proof that Artificial Neural Networks Overcome the Curse of Dimensionality in the Numerical Approximation of Black–Scholes Partial Differential Equations},
journal = {Mem. Am. Math. Soc.},
volume = {284},
number = {1410},
year = {2023},
}

@article{HAN2023106881,
title = {A new deep neural network algorithm for multiple stopping with applications in options pricing},
journal = {Commun. Nonlinear Sci. Numer. Simul.},
volume = {117},
pages = {106881},
year = {2023},
author = {Yuecai Han and Nan Li},
}

@article{Jia-Wong01022024,
author = {Bowen Jia and Hoi Ying Wong},
title = {Deep impulse control: application to interest rate intervention},
journal = {Quant. Finance},
volume = {24},
number = {2},
pages = {221--232},
year = {2024},
}

@article{gonon23,
  author  = {Lukas Gonon},
  title   = {Deep neural network expressivity for optimal stopping problems},
  journal = {Finance Stoch.},
  year    = {2024},
  volume ={28},
  pages   = {865–910},
}

@article{martyr16-discrete-switching,
 author = {R. MARTYR},
 journal = {Adv. Appl. Probab.},
 number = {3},
 pages = {832--847},
 title = {DYNAMIC PROGRAMMING FOR DISCRETE-TIME FINITE-HORIZON OPTIMAL SWITCHING PROBLEMS WITH NEGATIVE SWITCHING COSTS},
 volume = {48},
 year = {2016}
}

@misc{lin2009dual,
  author       = {Lin, B. and Ludkovski, M.},
  title        = {Dual Simulation Methods in Optimal Switching Problems},
  year         = {2009},
  month        = jul,
  howpublished = {Technical Report}
}

@misc{ye2025deepmartingale,
      title={DeepMartingale: Duality of the Optimal Stopping Problem with Expressivity and High-Dimensional Hedging}, 
      author={Junyan Ye and Hoi Ying Wong},
      year={2025},
      eprint={2510.13868},
      archivePrefix={arXiv},
      primaryClass={math.OC},
}

@article{Bayraktar23-deep-switching,
author = {Bayraktar, Erhan and Cohen, Asaf and Nellis, April},
title = {A Neural Network Approach to High-Dimensional Optimal Switching Problems with Jumps in Energy Markets},
journal={SIAM J. Financ. Math.},
volume = {14},
number = {4},
pages = {1028-1061},
year = {2023},
}

@article{puredual-mf,
author = {Alfonsi, Aurélien and Kebaier, Ahmed and Lelong, Jérôme},
title = {A Pure Dual Approach for Hedging Bermudan Options},
journal = {Math. Finance},
volume = {35},
number = {4},
pages = {745-759},
year = {2025}
}

@article{Roger02,
author = {Rogers, L.},
year = {2002},
pages = {271-286},
title = {Monte Carlo Valuing of American Options},
volume = {12},
journal = {Math. Finance},
type = {J}
}

@article{Haugh04,
author = {Kogan, L. and Haugh ,M.},
year = {2004},
pages = {258-270},
title = {Pricing American Options: A Duality Approach},
volume = {52},
number = {2},
journal = {Oper. Res.},
type={J}
}

@article{belome09,
author = {Belomestny, D. and Bender, C. and Schoenmakers, J.},
year = {2009},
pages = {53-71},
title = {True Upper Bounds For Bermudan Products Via Non-nested Monte Carlo},
volume = {19},
journal = {Math. Finance},
type = {J}
}

@article{roger10,
author = {Rogers, L.},
title = {Dual Valuation and Hedging of Bermudan Options},
journal={SIAM J. Financ. Math.},
volume = {1},
number = {1},
pages = {604-608},
year = {2010},
type={J}
}

@article{schoen13,
author = {Schoenmakers, J. and Zhang, J. and Huang, J.},
title = {Optimal Dual Martingales, Their Analysis, and Application to New Algorithms for Bermudan Products},
journal={SIAM J. Financ. Math.},
volume = {4},
number = {1},
pages = {86-116},
year = {2013},
type={J}
}

@article{Jentzen20,
author = {Hutzenthaler, M. and Jentzen, A. and Kruse, T. and et al.},
title = { A proof that rectified deep neural networks overcome the curse of dimensionality in the numerical approximation of semilinear heat equations},
journal = { SN Partial Differ. Equ. Appl.},
volume = {1},
number = {10},
year = {2020},
type = {J}
}

@article{Becker19,
  author  = {Becker, S. and Cheridito, P. and Jentzen, A.},
  title   = {Deep Optimal Stopping},
 journal={J. Mach. Learn. Res.},
  year    = {2019},
  volume  = {20},
  number  = {74},
  pages   = {1--25},
  type={J}
}

@article{Martyr-16-signed-switching,
author = {Martyr, Randall},
title = {Finite-Horizon Optimal Multiple Switching with Signed Switching Costs},
journal = {Math. Oper. Res.},
volume = {41},
number = {4},
pages = {1432-1447},
year = {2016},
}
\end{document}